\documentclass[a4paper,12pt, reqno]{amsart} 
\usepackage{amssymb,amsthm,amsmath}
\usepackage{multirow}
\usepackage{braket}

\usepackage{enumerate}
\usepackage{ifthen}
\usepackage{graphicx} 
\usepackage{tikz-cd} 
\usepackage[colorlinks,citecolor=blue,pagebackref,hypertexnames=false]{hyperref}

 \numberwithin{equation}{section}

\usepackage{amsfonts}
\usepackage{amscd}
\usepackage{amssymb}
\usepackage{enumerate}
\usepackage{graphicx}
\allowdisplaybreaks
\usepackage{mathabx}
\usepackage{color}
\usepackage{amsbsy}
\usepackage{graphicx}
\usepackage{amsthm}
\usepackage{amsmath}
\usepackage{amsxtra}
\usepackage{mathrsfs}
\usepackage{bbm}
\usepackage{dsfont}

\usepackage{ifthen}
\usepackage{xcolor,colortbl}
\newtheorem{theorem}{Theorem}[section]
\newtheorem{lemma}[theorem]{Lemma}
\newtheorem{problem}[theorem]{Problem}

\newtheorem{corollary}[theorem]{Corollary}
\newtheorem{proposition}[theorem]{Proposition}

\theoremstyle{definition}

\newtheorem{remark} [theorem]{Remark}

\theoremstyle{remark}

\numberwithin{equation}{section}
\definecolor{red}{rgb}{1.0, 0.0, 0.0}
\newcommand{\Bea}{\begin{eqnarray*}}
	\newcommand{\Eea}{\end{eqnarray*}}
\newcommand{\Be} {\begin{equation*}}
	\newcommand{\Ee} {\end{equation*}}
\newcommand{\be} {\begin{equation}}
	\newcommand{\ee} {\end{equation}}
\newcommand{\bea} {\begin{eqnarray}}
	\newcommand{\eea} {\end{eqnarray}}

	{\qed\bigskip}

	\newcounter{alphabet}

	\ifx\undefined\bysame
	\newcommand{\bysame}{\leavevmode\hbox to3em{\hrulefill}\,}
	\fi
	\usetikzlibrary{patterns}
\title[Orthonormal Strichartz estimates  on Wiener amalgam spaces]
{Orthonormal Strichartz estimates for the Schr\"odinger equations with Hamiltonian on Wiener amalgam spaces}

\author{Ramesh Manna, Shyam Swarup Mondal, Sarthak Tiwari}

\address{School of Mathematical Sciences, National Institute of Science Education and
	Research, Bhubaneswar, India.}
\address{Homi Bhabha National Institute, Training School Complex, Anushakti Nagar, Mumbai 400094, India}
\email{rameshmanna@niser.ac.in, sarthak.tiwari@niser.ac.in }

\address{ 
	\endgraf Stat-Math unit,\endgraf
	Indian Statistical Institute Kolkata, \endgraf
	BT Road,  Baranagar, Kolkata  700108, India}\endgraf
\email{mondalshyam055@gmail.com}


\date{\today}


\begin{document}
	
	\maketitle
	
	\begin{abstract}
		The main objective of this paper is to investigate orthonormal Strichartz estimates for Schr\"odinger equation on the Wiener amalgam space $\mathcal{W}(\mathcal{F} L^p, L^q)$.  More precisely, we first examine improvements in the time-integrability exponent of the existing Strichartz estimates established by Cordero and Nicola in \cite{NFC} for   Schr\"odinger equations associated with $\mathcal{H}^{+}=-\frac{1}{4\pi}\Delta+|x|^2$. We then extend these improved Strichartz estimates from a single initial datum to systems of orthonormal families, providing, to the best of our knowledge, the first results in this direction in the setting of Wiener amalgam spaces.
		%We then extend these improved Strichartz estimates from a single to a system of orthonormal families of initial data. 
		We further extend the classical Strichartz estimates in Wiener amalgam spaces from a single initial datum to systems of orthonormal families of initial data for the Schrödinger equation associated with the operator $\mathcal{H}^{-}=-\frac{1}{4\pi}\Delta-|x|^2$ and 
		Hamiltonian operator of the form $\mathcal{H}_{\mathcal{A}} = -\frac{1}{4\pi} B \nabla \cdot \nabla$, where $\mathcal{A} = \begin{pmatrix} 0 & B \\ 0 & 0 \end{pmatrix} \in \mathrm{Sp}(d,\mathbb{R})$ with $B = B^*$ and $\det B \neq0$.  %A key ingredient of our approach is the Stein complex interpolation theory for Wiener amalgam spaces $\mathcal{W}(\mathcal{F}L^p, L^q)$. 
		A key ingredient of our approach is Stein's complex interpolation theory for Wiener amalgam spaces, combined with a duality argument inspired by the work of Frank and Sabin.
		As an application of these orthonormal estimates associated with $\mathcal{H}^{+}, \mathcal{H}^{-}$, and $H_{\mathcal{A}},$ we establish local and small-data global well-posedness for the Hartree equation with infinitely many particles, for non-trace-class initial data.
	\end{abstract}
	\tableofcontents
	\allowdisplaybreaks
	\section{Introduction}
	
	The Schr\"odinger equation stands as one of the cornerstones of quantum mechanics and mathematical physics, governing the evolution of quantum states through the fundamental relation
	\[
	i \frac{\partial u}{\partial t} = H u, \qquad u(0, x) = u_0(x),
	\]
	where \(H\) denotes the Hamiltonian operator. For quadratic Hamiltonians, the time evolution operator \(e^{itH}\) admits an elegant description through the metaplectic representation, providing a rich interplay between symplectic geometry, harmonic analysis, and partial differential equations. This connection has attracted substantial attention in recent years, particularly through the pioneering work of Cordero and Nicola \cite{NFC}, who established Strichartz estimates for the Schr\"odinger equation with quadratic Hamiltonians within the framework of Wiener amalgam spaces.

	Wiener amalgam spaces, introduced by Feichtinger \cite{HGF}, offer a natural setting for time-frequency analysis by separately controlling the local regularity and global decay properties of functions,  providing a framework for studying functions beyond classical \(L^p\)-spaces. They give refined versions of classical Fourier inequalities and have applications in Sobolev embeddings and multiplier theory. Their extensions to modulation spaces have become fundamental in time-frequency analysis and have proved useful in several areas of harmonic analysis; see \cite{F}. Thus, Wiener amalgam spaces provide a framework for studying the localization, decay, and regularity of functions and operators; see \cite{F2026} and the references therein. Their flexibility has proven invaluable in studying dispersive equations, where solutions often exhibit different regularity and decay behaviors. 
    
    The metaplectic representation  $ 
	\mu : \operatorname{Sp}(d, \mathbb{R}) \to U(L^2(\mathbb{R}^d)),$ 
	originally constructed by Segal and Shale \cite{segal,SS}, provides the mathematical machinery to express the evolution operator for quadratic Hamiltonians as
	\[
	e^{it\mathcal{H}_{\mathcal{A}}} = \mu(e^{t\mathcal{A}}),
	\]
	where \(\mathcal{A}\) belongs to the symplectic Lie algebra \(\mathfrak{sp}(d, \mathbb{R})\).

	In this paper, we investigate the following free Schr\"odinger equation   associated with $\mathcal{H}_{\mathcal{A}}$, namely
	\[
	\begin{cases}
		i \partial_t u + \mathcal{H}_{\mathcal{A}} u = 0, \\
		u(0, x) = f(x),
	\end{cases}
	\]
	where \(\mathcal{H}_{\mathcal{A}}\) denotes the Weyl quantization of a quadratic form on \(\mathbb{R}^d \times \mathbb{R}^d\). Of particular interest are the harmonic oscillator
	\[
	\mathcal{H}^+ = -\frac{1}{4\pi} \Delta + \pi |x|^2,
	\]
	and the inverted harmonic oscillator
	\[
	\mathcal{H}^- = -\frac{1}{4\pi} \Delta - \pi |x|^2,
	\]
	and the quadratic Hamiltonian \(\mathcal{H}_{\mathcal{A}}\) corresponding to the symplectic matrix
	\[
	\mathcal{A} = \begin{pmatrix} 0 & B \\ 0 & 0 \end{pmatrix} \in \mathfrak{sp}(d, \mathbb{R}),
	\]
	where $B=B^*$ with $\det B \neq 0$. The general quadratic Hamiltonian $\mathcal{H}_{\mathcal{A}}$ corresponding to the symplectic matrix
	
	\[
	\mathcal{A} = \begin{pmatrix} A & B \\ C & D \end{pmatrix} \in \mathfrak{sp}(d, \mathbb{R}),
	\]
	
	explicitly takes the form
	\[
	\mathcal{H}_{\mathcal{A}} = -\frac{1}{4\pi} \sum_{j,k=1}^d B_{j,k} \frac{\partial^2}{\partial x_j \partial x_k}
	+ i \sum_{j,k=1}^d A_{j,k} x_k \frac{\partial}{\partial x_j}
	+ \frac{i}{2} \operatorname{Tr}(A)
	- \pi \sum_{j,k=1}^d C_{j,k} x_j x_k.
	\]
	
	The fixed-time estimates established by Cordero and Nicola \cite{NFC} demonstrate that solutions propagate regularity and decay across Wiener amalgam spaces in a controlled manner. For the harmonic oscillator $\mathcal{H}^+$,  the dispersive estimate
	\[
	\|{ e^{it\mathcal{H}^+} f}\|_{\mathcal{W}(\mathcal{F} L^1, L^\infty)}
	\lesssim |\sin t|^{-d} \|{ f }\|_{\mathcal{W}(\mathcal{F} L^\infty, L^1)},
	\]
	together with the \(TT^*\) argument and interpolation arguments yields the following Strichartz estimates on Wiener amalgam spaces (see Theorem 5.2, \cite{NFC}) 
	\begin{align}\label{CorNicostrichartz}
		\| e^{it\mathcal{H}^{+}} f \|_{L^{q/2}([0,T]) \mathcal{W}(\mathcal{F} L^{p'}, L^p)_x}
		\lesssim \|{ f }\|_{L^2_x},
	\end{align}
	for \(4 < q\le \infty\), \(2 \le p \le \infty\) and \(\frac{2}{q} + \frac{d}{p} = \frac{d}{2}\).
	In contrast, for the inverted harmonic oscillator $\mathcal{H}^-$, the corresponding dispersive estimate 
	\[
	\|{ e^{it\mathcal{H}^-} f }\|_{\mathcal{W}(\mathcal{F} L^1, L^\infty)}
	\lesssim \left( \frac{1 + |\sinh t|}{\sinh^2 t} \right)^{d/2}
	\|{ f }\|_{\mathcal{W}(\mathcal{F} L^\infty, L^1)},
	\]
	on Wiener amalgam spaces leads to the Strichartz estimate
	\begin{align}\label{cornicostrichartzH-}
		\|e^{it\mathcal{H}^-}f\|_{\mathcal{W}(L^{q/2},L^2)\mathcal{W}(\mathcal{F}L^{p'},L^p)}\lesssim\|f\|_{L^2},
	\end{align}
	under the same conditions \(4 < q\le \infty\), \(2 \le p \le \infty\) and \(\frac{2}{q} + \frac{d}{p} = \frac{d}{2}\).

	Let \(B\) is a symmetric nonsingular matrix with eigenvalues \(\lambda_1, \dots, \lambda_d\).  For the anisotropic free Hamiltonian $ 
	\mathcal{H}_{\mathcal{A}}= -\frac{1}{4\pi} \nabla \cdot B \nabla,$  the dispersive estimate on Wiener amalgam spaces is given by
	
	\begin{align}\label{dispersiveestimateHAcase}
		\|{ e^{it\mathcal{H}_{\mathcal{A}}} f }\|_{\mathcal{W}(\mathcal{F} L^1, L^\infty)}
		\lesssim \prod_{j=1}^d \left( \frac{1 + t^2 \lambda_j^2}{t^4 \lambda_j^4} \right)^{1/4}
		\|{ f }\|_{\mathcal{W}(\mathcal{F} L^\infty, L^1)}.  
	\end{align}

	Note that, in the isotropic case \(B = I_d\), the Hamiltonian reduces to the standard Laplacian $ 
	\mathcal{H}_{\mathcal{A}} = -\frac{1}{4\pi} \Delta$  on $\mathbb{R}^d.$ In the isotropic setting, Cordero and Nicola established Strichartz estimates on Wiener amalgam spaces in \cite{coni}.  Following the same arguments, we have the following Strichartz estimate for \(\mathcal{H}_{\mathcal{A}}\).
	\begin{theorem}[Strichartz estimate for $\mathcal{H}_{\mathcal{A}}$]\label{Single function estimates}
		Let $4 < q \leq \infty$,  $2 \leq p\leq \infty$ such that
		$ \frac{2}{q} + \frac{d}{p} = \frac{d}{2}.$ 
		Then  
		\begin{equation}\label{HA}
			\|e^{it\mathcal{H}_{\mathcal{A}}} f\|_{\mathcal{W}_t{(L^{q/2},L^q)} \mathcal{W}^{r'r}_x} \lesssim \|f\|_{L^2_x}.
		\end{equation}
	\end{theorem}
	
	It is worth noting that Cordero and Nicola \cite{NFC} established global-in-time dispersive estimates in Wiener amalgam spaces for the inverted harmonic oscillator via the metaplectic representation. However, in a recent work, Takizawa \cite{takizawa2025} generalized the Strichartz estimates obtained by Cordero and Nicola in \cite{NFC} to Schr\"odinger equations with general potentials of at most quadratic growth, employing a Hamiltonian flow approach rather than the metaplectic representation, and again obtained local-in-time estimates. We also refer to \cite{Wave} for Strichartz estimates for the wave equation in the setting of Wiener amalgam spaces.

	\subsection{Orthonormal Strichartz estimate}
	While the classical Strichartz estimates provide powerful control over single solutions, many physical phenomena, particularly in quantum statistical mechanics, Bose-Einstein condensation, and nonlinear optics, require understanding the collective behavior of systems of orthonormal functions. Such systems naturally arise when considering density matrices, coherent states, or fermionic and bosonic many-body quantum systems. In these contexts, the relevant quantities are not individual functions but rather orthonormal systems \(\{u_j\}_{j=1}^N\) with \(N\) potentially infinite, and the appropriate norms measure the combined behavior of the entire ensemble. 
	The idea of extending classical functional inequalities to orthonormal systems of input functions goes back to the famous work of Lieb-Thirring \cite{LT}, where a generalization of a certain Gagliardo-Nirenberg-Sobolev estimate to orthonormal systems of \(L^2\) functions was established. The Lieb-Thirring inequalities are a key component in the proof of stability of matter; see, for example, \cite{LT} or the comprehensive survey by Lieb \cite{EL} for further details.

	For $f\in L^2(\mathbb{R}^d)$, the classical Strichartz estimates for the free Schr\"odinger propagator \(e^{it\Delta}\) can be stated as
	\begin{equation}\label{Singlestrichartz}
		\|{ |e^{it\Delta} f|^2 }\|_{L^q_t L^p_x} \lesssim \|f\|_{L^2(\mathbb{R}^d)} 
	\end{equation}
	for all spatial dimensions \(d \ge 1\) and where \(p, q \ge 1\) satisfy $
	\frac{2}{q} + \frac{d}{p} = d$ and $ (q, p, d) \neq (1, \infty, 2).
	$ 
	
	The endpoint case is \((q, p) = \left(1, \frac{d}{d-2}\right)\) for \(d \ge 3\) has been proved by Keel and Tao in \cite{KT}, and all other allowable estimates follow by interpolation with the trivial estimate at \((q, p) = (\infty, 1)\). When \(d = 2\), the estimate fails at the endpoint \((q, p) = (1, \infty)\) (see, for example, \cite{montgomery-smith}) and for \(d = 1\), the estimate at \((q, p) = (2, \infty)\) is valid. Recently, the above classical Strichartz  estimate \eqref{Singlestrichartz} has been substantially generalized to the context of orthonormal systems \((f_j)_j\) in \(L^2(\mathbb{R}^d)\) in the work of Frank-Lewin-Lieb-Seiringer \cite{frank} and Frank-Sabin \cite{FS}, resulting in the following:
	\begin{theorem}\cite{frank, FS}
		Let  \(d \ge 1\) and \(p, q \ge 1\) satisfy
		\[
		\frac{2}{q} + \frac{d}{p} = d \quad \text{and}\quad \alpha = \frac{2p}{p+1}.\]
		Then for any orthonormal systems \((f_j)_j\) in \(L^2(\mathbb{R}^d)\) and all sequences \(\lambda = (\lambda_j)_j \in \ell^\alpha(\mathbb{C})\)
		\begin{align}\label{ONSLebegue}
			\left\| \sum_{j} \lambda_j |e^{it\Delta} f_j|^2 \right\|_{L^q_t L^p_x}
			\lesssim \left( \sum_{j} |\lambda_j|^{\alpha} \right)^{1/\alpha}.
		\end{align}
		The Schatten exponent is sharp in the sense that, for such \(p, q\), the estimate fails for all \(\alpha > \frac{2p}{p+1}\). Furthermore, when \(q = \frac{d+1}{d-1}\), the estimate (1.2) holds for all \(\alpha < \frac{2p}{p+1}\) and fails when \(\alpha = \frac{2p}{p+1}\).
	\end{theorem}

	The study of Fourier restriction problems in Schatten spaces has seen substantial progress through the work of Frank and Sabin. They established restriction estimates for smooth, compact surfaces in \(\mathbb{R}^n\) with non-vanishing Gaussian curvature, building on the classical Stein-Tomas theorem \cite{stein-thomas}. Using a duality argument, they then extended the Stein-Tomas estimate to orthonormal systems; see also \cite{frank-sabin2}. Their work also considered quadratic surfaces in \(\mathbb{R}^n\) within the framework of Schatten spaces. These results subsequently led to extensions of the classical Strichartz estimates \cite{RS97} from single functions to orthonormal systems for the Schr\"odinger, wave, and Klein-Gordon equations. Further developments in this direction can be found in \cite{BHLNS, BLN}.
	
	In recent years, Strichartz estimates for orthonormal systems have attracted considerable attention and have been studied in a variety of settings. In particular, much of the work has focused on orthonormal Strichartz estimates for Schr\"odinger propagators associated with different self-adjoint operators. Notable contributions include those of Mondal-Swain \cite{ssmjs} for the Hermite operator, Ghosh-Mondal-Swain \cite{gms} for the special Hermite operator, and Mondal-Song \cite{ms} for \((k,a)\)-generalized Laguerre operators and Jitendra-Pradeep-Mondal-Mejjaoli \cite{JMMP} for Fourier-Dunkl transform. Additionally, Nakamura \cite{nakamura} and Wang-Zhang-Zhang \cite{wzz} have treated the Laplace-Beltrami operator on compact manifolds. More general frameworks that apply to a broad class of dispersive equations have also been studied by Hoshiya \cite{Hos,hoshiya} and Feng-Song \cite{fs}. For recent developments in this direction, we refer to \cite{Partial} and \cite{Waveguide}, which extend orthonormal Strichartz estimates to the settings of partial regularity on the torus and waveguide manifolds, respectively. These developments rest on the foundational theory of convex functions of operators and the classical limit of quantum spin systems due to Berezin \cite{berezin} and Lieb \cite{lieb}, as well as the interpolation framework of Bergh and L\"ofstr\"om \cite{BER} and the trace ideal theory of Simon \cite{SIM}. 
	
	The orthonormal Strichartz inequality of Frank-Lewin-Lieb-Seiringer \cite{frank} and the weighted oscillatory integral estimates of Bez-Lee-Nakamura \cite{BLN} and Bez-Hong-Lee-Nakamura-Sawano \cite{BHLNS} further provide essential tools. In parallel, Wiener amalgam and time-frequency methods for the Schr\"odinger equation have been developed by Cordero-Nicola \cite{coni, NFC} and Cordero-Gr\"ochenig \cite{CGT}, building on Feichtinger's theory of modulation and Wiener amalgam spaces \cite{F, HGF, HFGB, HGFG, F2026}, while Triebel's interpolation theory \cite{TINT} supplies additional functional-analytic background.

	Motivated by the recent developments in the directions of orthonormal Strichartz estimates, here we aim to investigate Strichartz estimates for Schr\"odinger equations associated with several important operators (such as  $\mathcal{H}^{+}, \mathcal{H}^{-}$, and Hamiltonian operator $H_{\mathcal{A}}$), with the goal of extending the existing single-function estimates to orthonormal systems in the framework of Wiener amalgam spaces. To the best of our knowledge, orthonormal Strichartz estimates in Wiener amalgam spaces for the operators considered here have not yet been studied. We hope that our results will contribute to this developing area and provide a starting point for further investigations of orthonormal Strichartz estimates in more general function spaces and for a wider class of dispersive operators.

	%The goal of the present work is to extend the Strichartz theory for quadratic Schr\"odinger equations from the single-function setting to orthonormal systems in Wiener amalgam spaces.
	
	\subsection{Notational Conventions}
	Before proceeding, we fix the notation that will be used throughout this paper. 
	\begin{itemize} 
		\item  $\mathcal{H}_\mathcal{A}$:   the anisotropic Hamiltonian arising from  
		$ 
		\mathcal{A} = \begin{pmatrix} 0 & B \\ 0 & 0 \end{pmatrix},
		$ 
		with \(B^*=B\) and \(\det B \neq 0\).
		\item \( \mathcal{I} = (-\pi/4,\pi/4) \), unless otherwise specified.
		\item $\mathcal{H}^+  = -\frac{1}{4\pi} \Delta + \pi |x|^2: $ The  harmonic oscillator on $\mathbb{R}^d$.
		\item  $\mathcal{H}^- = -\frac{1}{4\pi} \Delta - \pi |x|^2:$ The inverted harmonic oscillator on $\mathbb{R}^d$.
		\item $\mathcal{W}^{p',p} := \mathcal{W}(\mathcal{F} L^{p'}, L^p):$ the Wiener amalgam space.
		\item \( L^q \mathcal{W}^{p',p} \):   the mixed space, where  the \( L^q \)-norm is taken with respect to the time variable \( t \), while the \( \mathcal{W}^{p',p} \)-norm is taken wrt the spatial variable \( x \).
		
		\item \( \mathcal{W}(L^r, L^s) \mathcal{W}^{p',p} \):   the mixed space,   the \( \mathcal{W}(L^r, L^s) \)-norm is with respect to the time variable \( t \), while the \( \mathcal{W}^{p',p} \)-norm is taken wrt the spatial variable \( x \).
		\item  \( A \lesssim B \): \( A \le C B \) for some constant \( C > 0 \) independent of the relevant parameters.
	\end{itemize}

	\section{Main results}
	This section presents the main results of the paper. We first note that the Strichartz estimate in \eqref{CorNicostrichartz} is somewhat weak, since the exponent in the time variable is only $\frac{q}{2}$. One of our main observations is that this exponent can be improved from $\frac{q}{2}$ to $q$. The following result gives the corresponding improved version of \eqref{CorNicostrichartz}.
	%Although the Strichartz estimate eqref{CorNicostrichartz} is rather weak in the sense that the time-variable exponent is merely $q/2$, we have nevertheless succeeded in sharpening this exponent to $q$. The following result is the improved version of the \eqref{CorNicostrichartz}.
	\begin{theorem}[Improved Strichartz estimate for $\mathcal{H}^{+}$]\label{strichartzH+improved}
		Let $2 < q \leq \infty$ and  $2 \leq p < \frac{2(d+1)}{d-1}$ such that
		\begin{equation}
			\frac{2}{q} + \frac{d}{p} = \frac{d}{2}.
		\end{equation}
		Then for $f\in L^2$, we have the following Strichartz estimates
		\begin{equation} 
			\|e^{it\mathcal{H}^{+}} f\|_{L^{q}_t((-\pi/4,\pi/4),\mathcal{W}^{p',p}_x(\mathbb{R}^d))} \lesssim \|f\|_{L^2_x}. \label{strichartzforH^+}
		\end{equation}
	\end{theorem}
	The following result extends the above improved estimate from the setting of a single function to a system of orthonormal functions.
	\begin{theorem}[ONS for $\mathcal{H}^{+}$]\label{ONSH+improved}
		Suppose $q \geq 1$ and $p\geq 1$ satisfies
		\[
		1 \leq p < \frac{d+1}{d-1} \quad \text{and}\quad \frac{2}{q} + \frac{d}{p} = d.
		\]
		Then for any orthonormal systems \((f_j)_j\) in \(L^2(\mathbb{R}^d)\) and  any sequence \(\{\alpha_j\}_{j\in J}\) in \(\mathbb{C}\), we have 
		\begin{align}
			\Bigg\|\sum_{j\in J}\alpha_j |e^{it\mathcal{H}^{+}}f_j|^2\Bigg\|_{L_t^{{q}}((-\pi/4,\pi/4), \mathcal{W}_x^{p',p}(\mathbb{R}^d))}\leq C_{d,p}\|\gamma\|_{\mathfrak{S}^{\frac{2p}{p+1}}(L^2(\mathbb{R}^d))}.\label{ONShermite1}
		\end{align} 
	\end{theorem}
	Here $\gamma$ is an operator on $L^2(\mathbb{R}^d)$ defined in \eqref{defgamma} and  $\|\gamma\|_{\mathfrak{S}^{\frac{2p}{p+1}}(L^2(\mathbb{R}^d))}=\|(\alpha_j)_{j\in J}\|_{l^{\frac{2p}{p+1}}}$. In addition to the above, we shall establish in the Theorem \ref{optimalityschattenexponent}, that the   Schatten exponent $\frac{2p}{p+1}$ on the right-hand side of the estimate \eqref{ONShermite1} is optimal. Furthermore,  the endpoint $\frac{d+1}{d-1}$ is sharp in the sense that the inequality \eqref{ONShermite1}  does not hold true for $p= \frac{d+1}{d-1}$ (See proposition \ref{endpointproposition}). Corresponding to the Strichartz estimates \eqref{CorNicostrichartz} on Wiener amalgam spaces, we obtain the following analogue for systems of orthonormal functions.	
	\begin{theorem}[ONS for the harmonic oscillator corresponds to Theorem 5.2 of \cite{NFC}]\label{ONSq/2}
		Suppose $q \geq 1$ and $p\geq 1$ satisfies
		\[
		1 \leq p < \frac{2d+1}{2d-1} \quad \text{and}\quad \frac{2}{q} + \frac{d}{p} = d.
		\]
		Then for any orthonormal systems \((f_j)_j\) in \(L^2(\mathbb{R}^d)\) and  any sequence \(\{\alpha_j\}_{j\in J}\) in \(\mathbb{C}\), we have 
		\begin{align}
			\Bigg\|\sum_{j\in J}\alpha_j |e^{it\mathcal{H}^{+}}f_j|^2\Bigg\|_{L_t^{\frac{q}{2}}((-\pi/4,\pi/4), \mathcal{W}_x^{p',p}(\mathbb{R}^d))}\leq C_{d,p}\|\gamma\|_{\mathfrak{S}^{\frac{2p}{p+1}}(L^2(\mathbb{R}^d))}.\label{ONShermite2}
		\end{align} 
	\end{theorem}

	Theorem \ref{Single function estimates} establishes the Strichartz estimates for $\mathcal{H}_{\mathcal{A}}$ in the setting of a single function. In the following result, we extend these estimates to systems of orthonormal functions in the one-dimensional case $d=1$.
	\begin{theorem}[ONS for $\mathcal{H}_{\mathcal{A}}$]\label{theoremONSHA}
		Suppose $p, q \geq 1$ satisfies
		\[
		1 \leq p < 5 \quad \text{and}\quad \frac{2}{q} + \frac{1}{p} = 1.
		\]
		Then for any orthonormal systems \((f_j)_j\) in \(L^2(\mathbb{R})\) and  any sequence \(\{\alpha_j\}_{j\in J}\) in \(\mathbb{C}\), we have 
		\begin{align}\label{ON1}
			\Bigg\|\sum_{j\in J}\alpha_j |e^{it\mathcal{H}_{\mathcal{A}}}f_j|^2\Bigg\|_{\mathcal{W}_t(L^{\frac{q}{2}},L^q)(\mathbb{R},  \mathcal{W}_x^{p',p}(\mathbb{R}))} \leq C_{d,p}\|\gamma\|_{\mathfrak{S}^{\frac{2p}{p+1}}(L^2(\mathbb{R}))}.%\label{ONSharmiteineq}   
		\end{align}
	\end{theorem}

	\begin{remark}
		Observe that, in the isotropic case \(B = I_d\), the Hamiltonian $\mathcal{H}_{\mathcal{A}}$ reduces to the standard Laplacian $ \mathcal{H}_{\mathcal{A}} = -\frac{1}{4\pi} \Delta$  on $\mathbb{R}^d.$ In this case, the above orthonormal inequality \eqref{ON1} extends the classical Strichartz estimates for $\Delta$ on the  Wiener amalgam spaces, established by Cordero and Nicola in  \cite{coni}.
	\end{remark}
	\begin{remark}
		The restriction \(d = 1\) comes from the fact that, in the time variable, we work with a Wiener amalgam space $\mathcal{W}(L^{q/2},L^q)$ rather than a Lebesgue space $L^q$ as in the harmonic oscillator case. To match the single-function estimate, we were compelled to use the weak decay of the kernel (see \eqref{kernelweakdecay}), which coincides with the decay of the dispersive estimate  \eqref{dispersiveestimateHAcase}. This constraint is what forces the argument to work only in dimension one.
	\end{remark}
	
	\begin{remark}
		Owing to the weak decay of the kernel (see \eqref{kernelweakdecay}), we believe that the range of \(p\) obtained above is not sharp; that is, the estimate is expected to hold for larger values of \(p\) as well. 
	\end{remark}

	In the following result, we extend the Strichartz estimate \eqref{cornicostrichartzH-} for $\mathcal{H}^-$ to the setting of systems of orthonormal functions.
	\begin{theorem}[ONS estimate for  $\mathcal{H}^{-}$]\label{theoremONSH-}
		Suppose $p, q \geq 1$ satisfy
		\[
		1 \leq p < 5 \quad \text{and}\quad \frac{2}{q} + \frac{1}{p} = 1.
		\]
		Then for any orthonormal systems \((f_j)_j\) in \(L^2(\mathbb{R})\) and  any sequence \(\{\alpha_j\}_{j\in J}\) in \(\mathbb{C}\), we have 
		\begin{align*}
			\Bigg\|\sum_{j\in J}\alpha_j |e^{it\mathcal{H}^{-}}f_j|^2\Bigg\|_{\mathcal{W}_t(L^{\frac{q}{2}},L^2)(\mathbb{R},  \mathcal{W}_x^{p',p}(\mathbb{R}))} \leq C_{d,p}\|\gamma\|_{\mathfrak{S}^{\frac{2p}{p+1}}(L^2(\mathbb{R}))}.%\label{ONSharmiteineq}   
		\end{align*}
	\end{theorem}
	
	\iffalse \begin{remark}
		Takizawa \cite{takizawa2025} generalized the Wiener amalgam Strichartz estimates of Cordero-Nicola \cite{NFC} to potentials of at most quadratic growth, but on a local-in-time interval. We note that the time-variable space appearing in his result is \(L^{q/2}\) (see Theorem 1.3 of \cite{takizawa2025}) rather than a Wiener amalgam space; this is a consequence of the finite time interval, which permits the Wiener amalgam norm in time to be replaced by the Lebesgue norm \(L^{q/2}([-T,T])\). Using an argument analogous to ours, one can obtain the corresponding ONS estimate for the inverted harmonic oscillator for a finite time interval $[-T,T]$.
	\end{remark}
	\begin{remark}
		Note also that the finite time interval gives rise to the Lebesgue space \(L^{q/2}\) in \eqref{CorNicostrichartz}, while our method improves the exponent from \(q/2\) to \(q\) (see Theorem \ref{strichartzH+improved}). The same reasoning shows that the exponent for the inverted harmonic oscillator on a finite interval can likewise be improved from \(q/2\) to \(q\).
	\end{remark}
	\fi

	Consider a system of \(M\) coupled nonlinear Hartree equations governing the dynamics of \(M\) fermions interacting through a potential \(w\), given by
	\[
	\begin{cases}
		i \partial_t u_1 = (L + w * \rho) u_1, & u_1|_{t=0} = f_1, \\
		\quad \vdots \\
		i \partial_t u_M = (L + w * \rho) u_M, & u_M|_{t=0} = f_M,
	\end{cases}
	\]
	where \((t,x) \in \mathcal{I} \times \mathbb{R}^{d}\) and $L$ is a self-adjoint operator acting on $L^2( \mathbb{R}^{d})$, \(\{f_j\}_{j=1}^M\) constitutes an orthonormal system in \(L^2( \mathbb{R}^{d})\), and \(\rho\) denotes the density function as $ 
	\rho(t,x) = \sum_{j=1}^M |u_j(t,x)|^2
	$. A fundamental inquiry pertains to the asymptotic behavior of solutions to  the above system in the limit \(M \to \infty\). In this regime, the formulation naturally translates into the following operator-valued evolutionary equation:
	\begin{align}\label{hartree}
		\begin{cases}
			i \partial_t \gamma = [L + w * \rho_\gamma, \gamma], \\
			\gamma|_{t=0} = \gamma_0.
		\end{cases}
	\end{align}
	Here, both \(\gamma_0\) and \(\gamma = \gamma(t)\) are bounded, self-adjoint operators acting on \(L^2( \mathbb{R}^{d})\).

	As an application of the orthonormal extension discussed above, we investigate the well-posedness of the Hartree equation \eqref{hartree} for an infinite number of particles, where the operator $L$ may be any one of $\mathcal{H}^+,\mathcal{H}^-$ or $\mathcal{H}_{\mathcal{A}}$.

	\subsection{Novelties and Difficulties}
	\begin{itemize}
		\item To prove the ONS estimates for operators $\mathcal{H}^{+},\mathcal{H}^{-}$ and $\mathcal{H}_{\mathcal{A}}$, we use two major results; the complex interpolation and the duality principle. The complex interpolation is established between the Schatten spaces and the mixed Wiener amalgam spaces. The statement of the complex interpolation is given in Lemma \ref{interpolationtheorem2}. However, we prove the duality principle in the mixed Lebesgue-Wiener amalgam setting in Theorem \ref{Duality}. Both of these results play an instrumental role in the whole theory of ONS estimates.\vspace{0.2cm}

		\item To prove ONS estimate for $\mathcal{H}_{\mathcal{A}}$ given in Theorem \ref{theoremONSHA}, a key feature of our method is the use of complex interpolation between two endpoint inequalities: the Schatten-\(2\) estimate (see \eqref{schatten2inequalityHA23}, where we use weak estimate of the kernel \eqref{kernelweakdecay}) and the Schatten-\(\infty\) estimate (see \eqref{infinitynormofWTW}). The Schatten-\(2\) endpoint is particularly delicate, as the time variable must be measured in a Wiener amalgam space \(\mathcal{W}(L^{q_1}, L^{q_2})\) whose exponents \(q_1\) and \(q_2\) are not independent. To align our interpolated estimate with the corresponding single-function Strichartz estimate, we must select \(q_1\) and \(q_2\) so that the Schatten-\(2\) inequality remains valid while simultaneously ensuring that after interpolation the resulting admissible condition coincides with the one obtained in the single-function setting. In our approach, we use embedding of Wiener amalgam spaces in the time variable to get desired \(q_1\) and \(q_2\). Then, we interpolate between the new Schatten-\(2\) inequality \eqref{Interpolation12} and the Schatten-\(\infty\) inequality \eqref{infinitynormofWTW} to obtain ONS for $\mathcal{H}_{\mathcal{A}}$. This requirement imposes a constraint on exponents that, after applying generalized Hardy-Littlewood-Sobolev inequality, Lemma \ref{youngtypeinequality}, can be satisfied only when \(d = 1\). The same follows for the case of the operator $\mathcal{H}^-$.\vspace{0.2cm}
		
		\item Since we use weak decay of the kernel in the proof of ONS estimate for $\mathcal{H}^-$ and $\mathcal{H}_{\mathcal{A}}$, we expect the range of the exponent $p$ in Theorem \ref{theoremONSHA} and Theorem \ref{theoremONSH-} can be improved.\vspace{0.2cm}
		\item  
		In Theorem \ref{strichartzH+improved}, we improved the single-function Strichartz estimate by replacing the exponent $\frac{q}{2}$ with $q$ in the time variable. Following a similar argument,  
		we can get the following version of the Strichartz estimates for  $L=\mathcal{H}_{\mathcal{A}}$ and $L=\mathcal{H}^{-}$
		\begin{equation}
			\|e^{itL} f\|_{L^q \mathcal{W}^{p',p}_x} \lesssim \|f\|_{L^2_x},\label{strichartzforL}
		\end{equation}
		where the time variable is measured in the Lebesgue space instead of the Wiener amalgam space for a certain range of $q$ and $p$. Since $L^q=\mathcal{W}(L^q,L^q)\hookrightarrow \mathcal{W}(L^{q/2},L^q)$ for every $2\leq q\leq \infty$, the estimate \eqref{strichartzforHA2} provides an improvement over Theorem \eqref{thm:strichartzz} for the operator $\mathcal{H}_{\mathcal{A}}$.  Consequently, we can also prove the corresponding orthonormal version of \eqref{strichartzforHA2}, that holds in every dimension $d\geq 1.$ \vspace{0.2cm}

		\item Takizawa \cite{takizawa2025} generalized the Wiener amalgam Strichartz estimates of Cordero-Nicola \cite{NFC} to potentials of at most quadratic growth, but on a local-in-time interval. We note that, the time-variable space appearing in his result is \(L^{q/2}\) (see Theorem 1.3 of \cite{takizawa2025}) rather than a Wiener amalgam space \(\mathcal{W}(L^{q_1}, L^{q_2})\); this is a consequence of the finite time interval, which permits the Wiener amalgam norm in time to be replaced by the Lebesgue norm \(L^{q/2}([-T,T])\). While our method improves the exponent from \(q/2\) to \(q\) (see Theorem \ref{strichartzH+improved}), the same reasoning shows that the exponent for the inverted harmonic oscillator on a finite interval can likewise be improved from \(q/2\) to \(q\) and subsequently can be extend for orthonormal system as well.

	\end{itemize}

	%We shall frequently encounter mixed space-time norms defined on the product domain \( \mathcal{I} \times \mathbb{R}^d \), where \( \mathcal{I} \subset \mathbb{R} \) denotes a non-empty open interval. In the case of the harmonic oscillator $\mathcal{H}^+$, the time interval is taken to be \( \mathcal{I} = (-\pi/4,\pi/4) \), whereas for the inverted harmonic oscillator \( \mathcal{H}^- \) and \(\mathcal{H}_{\mathcal{A}} \), we work with the full real line \( \mathcal{I} = \mathbb{R} \).We adopt the shorthand[\mathcal{W}^{p',p} := \mathcal{W}(\mathcal{F} L^{p'}, L^p),\] where the amalgam space is always understood to be defined on \( \mathbb{R}^d \). For instance, whenever we write \( L^q \mathcal{W}^{p',p} \) or \( \mathcal{W}(L^r, L^s) \mathcal{W}^{p',p} \), it is implicitly understood that the \( L^q \)-norm or the \( \mathcal{W}(L^r, L^s) \)-norm is taken with respect to the time variable \( t \), while the \( \mathcal{W}^{p',p} \)-norm is taken with respect to the spatial variable \( x \).sometimes, to avoid ambiguity in mixed norms, we shall write \( \| F \|_{\mathcal{W}_x} \) to indicate that the norm is taken with respect to the space variable \( x \), with \( t \) fixed. We shall employ the notation \( \lesssim \) to indicate that the inequality holds up to a multiplicative constant, i.e., \( A \lesssim B \) means \( A \le C B \) for some constant \( C > 0 \) independent of the relevant parameters.

	Apart from the introduction, the paper is organized as follows: In section \ref{prelim}, we review the definitions of Wiener amalgam spaces, the metaplectic representation, and the symplectic group, along with several important related results. We also recall the definition of Schatten class operators and complex interpolation in Schatten spaces. Section \ref{sectionwelldefined} contains the complex interpolation theorem and the duality principle in the setting of Wiener amalgam spaces. Section \ref{ONSH+} is devoted to proving improved Strichartz estimates and their orthonormal counterparts for the harmonic oscillator. We also establish the optimality and endpoint results. In Section \ref{sectionHA}, we prove the orthonormal Strichartz estimate for \(\mathcal{H}_{\mathcal{A}}\). Section \ref{sectionH-} contains the orthonormal Strichartz estimate for \(\mathcal{H}^-\). Finally, we investigate the well-posedness of the Hartree equation in Section \ref{application}.

	\section{Preliminaries}\label{prelim}
	In this section, we review some basic definitions and fundamental properties of Wiener amalgam spaces, the metaplectic representation, and the action of metaplectic operators on Wiener amalgam spaces. We also recall the Hermite polynomials and the Hermite semigroup, together with the basic properties of the kernel associated with the Hermite semigroup that will be required in the subsequent sections. For a detailed treatment of metaplectic operators and Wiener amalgam spaces, we refer the reader to \cite{HGF},\cite{HFGB},\cite{HGFG},\cite{FSA},\cite{HFK}. We begin by recalling the definition of the Wiener amalgam spaces $\mathcal{W}(\mathcal{F}L^p,L^q)$.

	\subsection{Wiener amalgam spaces} A Wiener amalgam space \(W(B, L^q)\), where typically \(B = L^p\) or \(B = \mathcal{F}L^p\), consists of functions that exhibit the local regularity of a function in \(B\) while displaying a global \(L^q\) decay. To define this, we fix a test function $g \in C_0^\infty$ and parameters $1 \le p, q \le \infty$. The Wiener amalgam space $\mathcal{W}(L^p, L^q)$, whose local component is $L^p$ and global component is $L^q$, is defined as the collection of all functions or tempered distributions $f$ for which
	\begin{align}
		\|f\|_{\mathcal{W}(L^p, L^q)} := \left\| \| f \, T_x g \|_{L^p} \right\|_{L^q_x} < \infty,
	\end{align}
	where the translation operator is given by $T_x g(t) := g(t-x)$. Similarly, we can define norm on the space $\mathcal{W}(\mathcal{F}L^p, L^q)$ as
	\begin{align*}
		\|f\|_{\mathcal{W}(\mathcal{F}L^p, L^q)} := \left\| \| f \, T_x g \|_{\mathcal{F}L^p} \right\|_{L^q_x} < \infty,
	\end{align*}
	where $\|h\|_{\mathcal{F}L^p}:=\|\mathcal{F}h\|_{L^p}$.
	Given a measurable function $f$ in variables $(t,x)$,  we shall use the following definition of mixed $L^q\mathcal{W}^{r,s}$
	$$\|f\|_{L^q(\mathcal{W}^{r,s})}=\bigg\|\|f\|_{\mathcal{W}^{r,s}}\bigg \|_{L^q}.$$
	
	The proof of the following properties of Wiener amalgam spaces can be found in \cite{HGF,HFGB,HGFG,H}.
	%[Cordero and Nicola, \cite{NFC}]
	\begin{lemma} \label{waspaceproperties}
		Let \(B_i, C_i\), \(i = 1, 2, 3\), be Banach spaces such that \(\mathcal{W}(B_i, C_i)\) are well defined. Then:
		
		\begin{enumerate}
			\item[(i)] (Convolution) If \(B_1 * B_2 \hookrightarrow B_3\) and \(C_1 * C_2 \hookrightarrow C_3\), we have
			\begin{equation}\nonumber
				\mathcal{W}(B_1, C_1) * \mathcal{W}(B_2, C_2) \hookrightarrow \mathcal{W}(B_3, C_3). 
			\end{equation}
			In particular, for every \(1 \le p, q \le \infty\), we have
			\begin{equation}\nonumber
				\|{ f * u }\|_{\mathcal{W}(\mathcal{F} L^p, L^q)}
				\lesssim \|{ f }\|_{\mathcal{W}(\mathcal{F} L^\infty, L^1)} \|{ u }\|_{\mathcal{W}(\mathcal{F} L^p, L^q)}. 
			\end{equation}
			
			\item[(ii)] (Inclusions) If \(B_1 \hookrightarrow B_2\) and \(C_1 \hookrightarrow C_2\), then
			\begin{equation}\nonumber
				\mathcal{W}(B_1, C_1) \hookrightarrow \mathcal{W}(B_2, C_2).
			\end{equation}
			Moreover, the inclusion of \(B_1\) into \(B_2\) need only hold ``locally'' and the inclusion of \(C_1\) into \(C_2\) ``globally.'' In particular, for \(1 \le p_i, q_i \le \infty\), \(i = 1, 2\), we have
			\[
			p_2 \le p_1 \quad \text{and} \quad q_1 \le q_2 \quad \Longrightarrow \quad \mathcal{W}(L^{p_1}, L^{q_1}) \hookrightarrow \mathcal{W}(L^{p_2}, L^{q_2}). 
			\]

			\item[(iii)] (Complex interpolation) For \(0 < \theta < 1\), we have
			\[
			\big[ \mathcal{W}(B_1, C_1), \mathcal{W}(B_2, C_2) \big]_{[\theta]}
			= \mathcal{W}\big( [B_1, B_2]_{[\theta]}, [C_1, C_2]_{[\theta]} \big),
			\]
			if \(C_1\) or \(C_2\) has absolutely continuous norm.
			
			\item[(iv)] (Duality) If \(B'\), \(C'\) are the topological dual spaces of the Banach spaces \(B\), \(C\), respectively, and the space of test functions \(C_0^\infty\) is dense in both \(B\) and \(C\), then
			\[
			\mathcal{W}(B, C)' = \mathcal{W}(B', C').\]

		\end{enumerate}
		
		\begin{proposition}\label{equivalenceofWL^pspaces}
			For any $1\leq p \leq \infty$, Banach spaces $\mathcal{W}(L^p,L^p)$ and $L^p$ are equivalent.
		\end{proposition}
	\end{lemma}
	\begin{proposition}[\cite{CGT}, Proposition 2.4]  \label{algebraproperty}
		For every \(1 \le p, q \le \infty\), we have
		\[
		\|{ fu }\|_{\mathcal{W}(\mathcal{F} L^p, L^q)}
		\lesssim \|{ f }\|_{\mathcal{W}(\mathcal{F} L^1, L^\infty)} \|{ u }\|_{\mathcal{W}(\mathcal{F} L^p, L^q)}.
		\] 
		
	\end{proposition}
	Of a central importance to our forthcoming developments is the result stated below, which generalizes the classical Hardy–Littlewood–Sobolev fractional integration inequality (see e.g. \cite[p. 119]{SSIN}). The latter is recovered as the special case wherein the convolution operator acts via the kernel \(K(x) = |x|^{-\alpha} \in L^{d/\alpha, \infty}\), where \(0 < \alpha < d\).

	\begin{proposition}[\cite{TINT}, Theorem 2]\label{generalyoungstheorem}
		
		Let \(1 \le p < q < \infty\), \(0 < \alpha < d\), with
		\[
		\frac{1}{p} = \frac{1}{q} + 1 - \frac{\alpha}{d}.
		\]
		Then,
		\[
		L^p(\mathbb{R}^d) * L^{d/\alpha, \infty}(\mathbb{R}^d) \to L^q(\mathbb{R}^d).
		\]
	\end{proposition}

	For \(0 < \alpha < 1/2\), let $\psi_{\alpha}$ and $\varphi_{\alpha}$ be nonnegative real functions defined by
	
	\begin{equation}
		\psi_{\alpha}(t) = |t|^{-\alpha} + |t|^{-2\alpha}, \qquad t \in \mathbb{R}, \quad t \neq 0, \label{Phidefn.}
	\end{equation}
	and 
	\begin{align}\label{vaphidef.}
		\varphi_\alpha(t) = |\sinh t|^{-\alpha} + |\sinh t|^{-2\alpha}, \qquad t \in \mathbb{R}, \quad t \neq 0.
	\end{align}

	\begin{lemma}[\cite{coni}, Lemma 4.1]\label{phispispaces}
		We have
		\[
		\psi_\alpha \in \mathcal{W}\left( L^{1/(2\alpha), \infty}, L^{1/\alpha, \infty} \right) \ \ \  and \ \ \ \varphi_\alpha \in \mathcal{W}\left( L^{1/(2\alpha), \infty}, L^1 \right).
		\]
	\end{lemma}

	\begin{lemma}\label{youngtypeinequality}
		For \(0 < \alpha < 1/2\), let \(\psi_{\alpha}\) and \(\varphi_\alpha\) be the function defined in 
		\eqref{Phidefn.} and \eqref{vaphidef.} respectively. Then,
		\[
		\| f * \psi_\alpha \|_{\mathcal{W}(L^{1/\alpha}, L^{2/\alpha})}
		\lesssim \| f \|_{\mathcal{W}(L^{(1/\alpha)'}, L^{(2/\alpha)'})} \quad \textit{and}\quad \| f * \varphi_\alpha \|_{\mathcal{W}(L^{1/\alpha},L^2)}\lesssim\|f\|_{\mathcal{W}(L^{(1/\alpha)'},L^2)}.
		\]
	\end{lemma}
	For the proof of the first inequality, the readers are referred to \cite[Lemma 4.2]{coni}. Using the similar argument, one can also prove the second inequality, see \cite{NFC}.

	\subsection{Metaplectic representation and symplectic group}
	
	The symplectic group is defined by
	\[
	\operatorname{Sp}(d, \mathbb{R}) = \left\{ g \in \operatorname{GL}(2d, \mathbb{R}) : {}^t g J g = J \right\},
	\]
	where the symplectic matrix \(J\) is  given by $$J=\left[\begin{array}{cc}0 & I_d \\ -I_d & 0\end{array}\right].$$
	
	The metaplectic or Shale--Weil representation \(\mu\) is a unitary representation of the (double cover of the) symplectic group \(\operatorname{Sp}(d, \mathbb{R})\) on \(L^2(\mathbb{R}^d)\). For elements of \(\operatorname{Sp}(d, \mathbb{R})\) in special form, the metaplectic representation can be computed explicitly. For \(f \in L^2(\mathbb{R}^d)\), we have
	\[
	\mu \begin{pmatrix} A & 0 \\ 0 & {}^t A^{-1} \end{pmatrix} f(x)
	= (\det A)^{-1/2} f(A^{-1}x),
	\]
	and
	\[
	\mu \begin{pmatrix} I_d & 0 \\ C & I_d \end{pmatrix} f(x)
	= \pm e^{i\pi \langle C x, x \rangle} f(x). 
	\]
	
	The symplectic algebra \(\mathfrak{sp}(d, \mathbb{R})\) is the set of all \(2d \times 2d\) real matrices \(A\) such that \(e^{tA} \in \operatorname{Sp}(d, \mathbb{R})\) for all \(t \in \mathbb{R}\).
	
	%{\color{red}The following formulae for the metaplectic representation can be found in [\cite{FHP},Theorems 4.51 and 4.53]. ??}

	The metaplectic representation $\mu: \mathrm{Sp}(d,\mathbb{R}) \to \mathcal{U}(L^2(\mathbb{R}^d))$ is a unitary representation of the symplectic group $\mathrm{Sp}(d,\mathbb{R})$ on $L^2(\mathbb{R}^d)$. It is defined as the unique (up to a phase) lift of the natural action of $\mathrm{Sp}(d,\mathbb{R})$ on the phase space $\mathbb{R}^{2d}$ to the Hilbert space of square-integrable functions.
	Suppose $1 \le p, q \le \infty$ and let $\mathcal{A} = \begin{pmatrix} A & B \\ C & D \end{pmatrix} \in \mathrm{Sp}(d,\mathbb{R})$ satisfy $\det B \neq 0$. Then the associated metaplectic operator $\mu(\mathcal{A})$ furnishes a continuous mapping from $\mathcal{W}(\mathcal{F} L^q, L^p)$ into $\mathcal{W}(\mathcal{F} L^p, L^q)$, meaning that
	\begin{align}
		\|\mu(\mathcal{A}) f\|_{\mathcal{W}(\mathcal{F} L^p, L^q)} \le \alpha(\mathcal{A},p,q) \|f\|_{\mathcal{W}(\mathcal{F} L^q, L^p)}.
	\end{align}
	The constant $\alpha = \alpha(\mathcal{A},p,q)$ appearing in the above estimate admits an explicit formulation in terms of the matrix entries of $A$ and the exponents $p$ and $q$.
	
	\begin{proposition}
		\label{propmetaplectic}
		Let $\mathcal{A} = \begin{pmatrix} A & B \\ C & D \end{pmatrix} \in \mathrm{Sp}(d,\mathbb{R})$.
		
		\begin{enumerate}
			\item[(i)] If $\det B \neq 0$, then the metaplectic operator acts on $f$ according to
			\begin{equation}
				\mu(\mathcal{A})f(x) = i^{d/2} (\det B)^{-1/2} \int e^{-\pi i x \cdot D B^{-1} x + 2\pi i y \cdot B^{-1} x - \pi i y \cdot B^{-1} A y} f(y) \, dy. 
			\end{equation}
			
			\item[(ii)] If $\det A \neq 0$, we have
			\begin{equation}
				\mu(\mathcal{A})f(x) = (\det A)^{-1/2} \int e^{-\pi i x \cdot C A^{-1} x + 2\pi i \xi \cdot A^{-1} x + \pi i \xi \cdot A^{-1} B \xi} \hat{f}(\xi) \, d\xi. 
			\end{equation}
		\end{enumerate}
	\end{proposition}
	\begin{proposition}
		\label{propmetaplecticalternative}
		Let $\mathcal{A} = \begin{pmatrix} A & B \\ C & D \end{pmatrix} \in \mathrm{Sp}(d,\mathbb{R})$ be such that both $\det B \neq 0$ and $\det A \neq 0$. Then the metaplectic operator admits the alternative representation
		\begin{align}\label{represetationofmuA}
			\mu(\mathcal{A})f(x) = (-i \det B)^{-1/2} e^{-\pi i x \cdot C A^{-1} x} \left( e^{-\pi i y\, \cdot B^{-1} Ay} * f \right)\!(A^{-1} x).
		\end{align}
		
	\end{proposition}

	Next, we recall the Hermite polynomials along with their fundamental properties and those of the kernel associated with the Hermite semigroup that will be used in the subsequent sections. For a detailed treatment of the Hermite operator and its corresponding semigroup, we refer the reader to \cite{NR2005,T}.
	\subsection{Harmonic oscillator(Hermite operator)}
	For \( k \in \mathbb{N}_0 \), the Hermite polynomials \( H_k(x) \) on \( \mathbb{R} \) are
	\[
	H_k(x)e^{-x^2} = (-1)^k \frac{d^k}{dx^k} (e^{-x^2}).
	\]
	
	Then the functions
	\[
	h_k(x) = \frac{(-1)^k}{\sqrt{2^kk\sqrt{\pi}!}} \left( \frac{d^k}{dx^k} e^{-x^2} \right) e^{\frac{x^2}{2}},
	\]
	are in \( L^2(\mathbb{R}) \) and $  h_k \|_{L^2} = 1 \quad \text{for } k \in \mathbb{N}_0.$ The tensor product
	\[
	h_\alpha(x) := \prod_{l=1}^d h_{\alpha_l}(x); \quad \alpha = (\alpha_1, \alpha_2, \dots, \alpha_d),
	\]
	are the eigenfunctions for Hermite operators $\mathcal{H}^{+}=-\Delta+|x|^2$ with eigenvalue \( 2|\alpha|+d \) on $\mathbb{R}^d$. They form a complete orthonormal system in \( L^2(\mathbb{R}^d) \). For each \( f \in L^2(\mathbb{R}^d) \), using spectral property, we have the following representation
	\[
	f = \sum_{\alpha} \langle f, h_\alpha \rangle h_\alpha = \sum_{k=0}^{\infty} P_k f,
	\]
	where \( P_k \) denotes the orthogonal projection of $L^2(\mathbb{R}^d)$ onto the eigenspace spanned by $\{h_\alpha:|\alpha|=k \}$, given by 
	\[P_kf(x)=\sum_{|\alpha|=k}\langle f, h_\alpha \rangle h_\alpha(x).\]
	Setting $ 
	\Phi_k(x,y) = \sum_{|\alpha| = k} h_\alpha (x) h_\alpha (y), 
	$ the Hermite projectio $P_k$ can be written as
	\[
	P_k f(x) = \int_{\mathbb{R}^d} \Phi_k(x,y) f(y) \, dy.
	\]

	Moreover, the operator $\mathcal{H}^+$ defines a semigroup called the Hermite semigroup
	$e^{-it \mathcal{H}^+}, t>0,$   defined by
	$$
	e^{-it \mathcal{H}^+} f=\sum_{k=0}^{\infty} e^{-(2 k+d) it} P_{k} f
	$$
	for $f \in L^{2}(\mathbb{R}^{d}).$ On a dense subspace, say the space of all Schwartz functions, the above can be written as
	$$
	e^{-it \mathcal{H}^+} f(x)=\int_{\mathbb{R}^{n}} f(y) K_{it}(x, y) d y,
	$$ where the kernel $K_{t}(x, y)$ is given by the expansion
	$$
	K_{it}(x, y)=\sum_{\alpha \in \mathbb{N}^{d}_0} e^{-(2|\alpha|+d) it} \Phi_{\alpha}(x) \Phi_{\alpha}(y)= \sum_{k=0}^{\infty} e^{-it(2k+d)} \Phi_k(x,y).
	$$

	For $z=r+i t, r>0, t \in \mathbb{R},$ the kernel of the operator $e^{-z\mathcal{H}^+}$ is given by
	\begin{align*}
		K_z(x, y) &= \sum_{k=0}^{\infty} e^{-z(2k+d)} \Phi_k(x, y) = \sum_{k=0}^{\infty} e^{-(r+it)(2k+d)} \Phi_k(x, y)\\
		&= e^{-d(r+it)} \sum_{k=0}^{\infty} e^{-2k(r+it)} \Phi_k(x, y)\\
		&=:e^{-d(r+it)} \sum_{k=0}^{\infty} w^k \Phi_k(x, y),
	\end{align*}
	with \( w = e^{-2(r+it)} \). Since $|w|<1$, thus by Mehler's formula
	\[
	K_{r+it}(x, y) = \pi^{-1/2} e^{-d(r+it)} (1-w^2)^{-1/2} \exp\left( -\frac{1}{2} \left(\frac{1+w^2}{1-w^2}\right) (|x|^2 + |y|^2) + \frac{2w}{1-w^2} x.y \right).
	\]

	For $t\in \mathbb{R}\setminus(\frac{\pi}{2})\mathbb{Z}$, letting $r\to 0$, the kernel of the operator $e^{-it\mathcal{H}^+}$ can be written as
	\begin{align}\label{CH2k}
		K_{it}(x, y)&=\frac{e^{-\frac{i\pi d}{4}}}{(2\pi\sin 2t)^{\frac{d}{2}}} e^{\frac{i}{2}\left( \cot 2t(|x|^2+|y|^2)-\frac{2 x\cdot y}{\sin  2t}\right) }.
	\end{align}

	%Consider the initial value problem for Schr\"odinger equation, Associated to Hermite operator $\mathcal{H}^{+}=-\Delta+|x|^2$
	%\begin{align}
	%    \begin{cases}
		%        i\frac{\partial u}{\partial t}+\mathcal{H}^{+}u =0,\ \ (x,t)\in\mathbb{R}^d\times(0,\infty)\\
		%        u(0,x)=f(x),
		%    \end{cases}\label{schrodingerinitialvalueeqn}
	%\end{align}
	%Solution of the equation \eqref{schrodingerinitialvalueeqn}
	%\begin{align*}
	%    e^{it\mathcal{H}^{+}} f(x) &= \sum_{k=0}^{\infty} e^{it(2k+d)} P_k f(x)\\
	%    &= \sum_{k=0}^{\infty} e^{it(2k+d)} \int_{\mathbb{R}^d} \Phi_k(x,y) f(y) \, dy\\
	%    &= \int_{\mathbb{R}^d} \left( \sum_{k=0}^{\infty} e^{it(2k+d)} \Phi_k(x,y) \right) f(y) \, dy\\
	%    &=:\int_{\mathbb{R}^d}  K_{it}(x,y) f(y) \, dy.
	%\end{align*}
	%Where 
	%\[
	%K_{it}(x,y) = \sum_{k=0}^{\infty} e^{it(2k+d)} \Phi_k(x,y).
	%\]
	%\bigskip
	The following lemma tells about the uniform decay of the kernel $K.$
	\begin{lemma} \cite{R2008}\label{Lemmaratnakumar}
		Let \( K_z(x, y) \) be as before with \( z = r + it \), \( r > 0 \), \( 0 < |t| \leq \pi \). Then
		\[
		\left| K_z(x, y) \right| \leq \frac{e^{-dr}}{|\sin(2t)|^{d/2}}.
		\]
		
	\end{lemma}
	If  we consider $t\in \mathcal{I}:=(-\frac{\pi}{4},\frac{\pi}{4})$, then  
	$
	\left| K_z(x, y) \right| \lesssim \frac{1}{|t|^{d/2}}
	$ for all $r>0$. This shows that for a smaller time interval, the decay of the Kernel of Hermite operator is dominated by polynomial decay.
	\subsection{Littlewood--Sobolev inequality}
	A fundamental tool for establishing Strichartz estimates is the Hardy--Littlewood--Sobolev inequality for Lebesgue spaces. However, when aiming for Strichartz estimates in Lorentz spaces, we require a more refined version of this inequality. According to [32], the following Littlewood--Sobolev inequality holds in the Lorentz space setting. Let $\sigma \in (0,1)$ and let $p_1, p_2, r_1, r_2 \in (1, \infty)$ satisfy
	
	\[
	\frac{1}{p_1} + \frac{1}{p_2} + \sigma = 2 \quad \text{and} \quad \frac{1}{r_1} + \frac{1}{r_2} \geq 1.
	\]
	
	Then there exists a constant $C > 0$ such that for all measurable functions $g$ and $h$ on $I$,
	
	\[
	\left| \int_I \int_I \frac{g(t) h(s)}{|t-s|^\sigma} \, ds \, dt \right| \leq C \, \|g\|_{L^{p_1, r_1}(I)} \|h\|_{L^{p_2, r_2}(I)}.
	\]

	\subsection{Schatten Class}
	
	We begin by briefly recalling the notion of Schatten spaces. For a thorough treatment, the reader may consult \cite{SIM}. Let \( H \) be a complex, separable Hilbert space, and let \( T : H \to H \) be a compact operator with adjoint \( T^* \). The singular values of \( T \), denoted by \( \{s_j(T)\}_{j \in \mathbb{N}} \), are defined as the nonzero eigenvalues of the positive operator \( |T| := \sqrt{T^*T} \). These values form an at most countable set. For any parameter \( \lambda \in [1, \infty) \), the Schatten space \( \mathfrak{S}^\lambda(H) \) consists of all compact operators \( T \) for which the sum of the \( \lambda \)-th powers of the singular values is finite; specifically,
	
	\[
	\sum_{j=1}^{\infty} s_j(T)^\lambda < \infty.
	\]
	
	Equipped with the norm
	
	\[
	\|T\|_{\mathfrak{S}^\lambda(H)} = \left( \sum_{j=1}^{\infty} s_j(T)^\lambda \right)^{\frac{1}{\lambda}},
	\]
	
	\( \mathfrak{S}^\lambda(H) \) becomes a Banach space. When \( \lambda = \infty \), the space \( \mathfrak{S}^\infty(H) \) is understood as the set of compact operators, though its norm coincides with the usual operator norm. An important duality relation holds: \( (\mathfrak{S}^\infty(H))^* \cong \mathfrak{S}^1(H) \) and \( (\mathfrak{S}(H))^* \cong B(H,H) \), where \( B(H,H) \) denotes the space of all bounded linear operators on \( H \). This behavior differs from that of standard \( L^p \) spaces. In the special cases, \( \mathfrak{S}^1(H) \) is called the \textbf{trace class}, and \( \mathfrak{S}^2(H) \) is called the \textbf{Hilbert-Schmidt class}. As we will see later, these Schatten spaces appear naturally when we prove Strichartz estimates for systems of orthonormal initial data, particularly through duality arguments.
	
	\subsection{Complex Interpolation in Schatten Spaces}
	
	Let \( a_0 < a_1 \) be real numbers. Consider a family of operators \( \{T_z\} \) defined on \( X \times I \) for complex numbers \( z \) lying in the closed strip \( a_0 \leq \operatorname{Re}(z) \leq a_1 \). Such a family is said to be \textbf{analytic in the sense of Stein} if it satisfies the following conditions:
	
	\begin{enumerate}
		\item For every \( z \) within the strip, \( T_z \) maps simple functions on \( X \times I \) to linear combinations of simple functions.
		\item For any pair of simple functions \( F \) and \( G \) on \( X \times I \), the complex function \( z \mapsto \langle G, T_z F \rangle \) is analytic in the interior of the strip (\( a_0 < \operatorname{Re}(z) < a_1 \)) and continuous on its closure.
		\item There exists a function \( C(s) \), growing at most doubly exponentially in \( s \), such that for all real \( s \),
		\[
		\sup_{\lambda \in [a_0, a_1]} |\langle G, T_{\lambda + i s} F \rangle| \leq C(s).
		\]
	\end{enumerate}
	
	This family of operators allows for a complex interpolation method adapted to Schatten spaces, which will be a key tool in proving our main results. For further details on complex interpolation theory, we refer the reader to
	\cite[Proposition 1]{FS},  \cite[Theorem 2.9]{SIM} , and  \cite[Theorem 2.26]{NG}.
	
	\section{Some Technical Lemmas}\label{sectionwelldefined}
	In this section, we establish several preliminary results that will be used throughout the subsequent analysis. %We first examine the well-definedness and boundedness of the multiplication operators on the appropriate Wiener amalgam spaces%. 
	We then establish a complex interpolation result for analytic families of operators acting on Wiener amalgam spaces. Finally, we formulate and prove a duality principle connecting Schatten class estimates with estimates for orthonormal systems. These technical results provide the necessary framework for the proofs of the main results in the subsequent sections. %We begin by establishing the well-definedness of the multiplication operators $W$ and $\overline{W}$.%

	\subsection{Complex interpolation in Wiener amalgam spaces}
	
	We state the following interpolation theorem on  Wiener amalgam spaces without proof. The interested readers are referred to proposition \ref{waspaceproperties} (property (iii)), Bergh and Löfström
	[Interpolation Spaces
	An Introduction, \cite{BER}] and Simon [Trace ideals and their applications, \cite{SIM}].
	\begin{lemma}\label{interpolationtheorem2}
		Let $\{T_z\}$ be an analytic family of operators on $ \mathbb{R}\times \mathbb{R}^d $ in the sense of Stein defined in the strip $a_0 \leq \operatorname{Re}(z) \leq a_1$. Let $1 \leq p_0, p_1,r_0,r_1,q_0, q_1,\tilde{q}_0,\tilde{q}_0 \leq \infty$ and for any $\theta \in [0,1]$, consider 
		$\frac{1-\theta}{p_0}+\frac{\theta}{p_1}=\frac{1}{p}$, $\frac{1-\theta}{q_0}+\frac{\theta}{q_1}=\frac{1}{q}$ and $\frac{1-\theta}{\tilde{q}_0}+\frac{\theta}{\tilde{q}_1}=\frac{1}{\tilde{q}}$. Suppose $W_j \in \mathcal{W}_t(L^{\tilde{q}},L^q)(\mathbb{R}, \mathcal{W}^{p',p}(\mathbb{R}^d)),~j=1,2$.  If there exists $C_0, C_1, b_0, b_1 > 0$
		%and $1 \leq p_0, p_1,r_0,r_1,q_0, q_1 \leq \infty$
		such that 
		\[
		\|W_1T_{a_0+is}W_2\|_{\mathfrak{S}^{r_0}(L^2(\mathbb{R}, L^2(\mathbb{R}^d)))} \leq C_0 e^{b_0 |s|} \|W_1\|_{\mathcal{W}_t(L^{\tilde{q_0}},L^{q_0})(\mathbb{R}, \mathcal{W}_x^{p'_0,p_0}(\mathbb{R}^d))} \|W_2\|_{\mathcal{W}(L^{\tilde{q_0}},L^{q_0})(\mathbb{R}, \mathcal{W}^{p'_0,p_0}(\mathbb{R}^d))}
		\]
		
		and
		
		\[
		\|W_1T_{a_1+is}W_2\|_{\mathfrak{S}^{r_1}(L^2(\mathbb{R}, L^2(\mathbb{R}^d)))} \leq C_1 e^{b_1 |s|} \|W_1\|_{\mathcal{W}_t(L^{\tilde{q_1}},L^{q_1})(\mathbb{R}, \mathcal{W}_x^{p'_1,p_1}(\mathbb{R}^d))} \|W_2\|_{\mathcal{W}_t(L^{\tilde{q_1}},L^{q_1})(\mathbb{R}, \mathcal{W}_x^{p'_1,p_1}(\mathbb{R}^d))} 
		\]

		Then, 
		%for any $\theta \in (0,1)$, 
		there exist constants $c$ such that 
		\[\|W_1T_{a+is}W_2\|_{\mathfrak{S}^{r}(L^2(\mathbb{R}, L^2(\mathbb{R}^d)))} \leq C_0^{1-\theta}C_1^{\theta} e^{c |s|} \|W_1\|_{\mathcal{W}_t(L^{\tilde{q}},L^{q})(\mathbb{R}, \mathcal{W}_x^{p',p}(\mathbb{R}^d))} \|W_2\|_{\mathcal{W}_t(L^{\tilde{q}},L^{q})(\mathbb{R}, \mathcal{W}_x^{p',p}(\mathbb{R}^d))}  
		,\]
		where $a=(1-\theta)a_0+\theta a_1$, and
		%$\frac{1-\theta}{p_0}+\frac{\theta}{p_1}=\frac{1}{p}$, 
		$\frac{1-\theta}{r_0}+\frac{\theta}{r_1}=\frac{1}{r}$. 
		%and $\frac{1-\theta}{q_0}+\frac{\theta}{q_1}=\frac{1}{q}$. 
	\end{lemma}

	\begin{corollary}\label{interpolationtheorem}
		Let $\{T_z\}$ be an analytic family of operators on $ \mathcal{I}\times \mathbb{R}^d $ in the sense of Stein defined in the strip $a_0 \leq \operatorname{Re}(z) \leq a_1$.  Consider $\mathcal{I}$ to be some nonempty interval of $\mathbb{R}$. Let $1 \leq p_0, p_1,r_0,r_1,q_0, q_1 \leq \infty$ and for any $\theta \in [0,1]$, consider 
		$\frac{1-\theta}{p_0}+\frac{\theta}{p_1}=\frac{1}{p}$, and $\frac{1-\theta}{q_0}+\frac{\theta}{q_1}=\frac{1}{q}$. Suppose $W_j \in L^{q}(\mathcal{I}, \mathcal{W}^{p',p}(\mathbb{R}^d)),~j=1,2$.  If there exists $C_0, C_1, b_0, b_1 > 0$
		%and $1 \leq p_0, p_1,r_0,r_1,q_0, q_1 \leq \infty$
		such that 
		\[
		\|W_1T_{a_0+is}W_2\|_{\mathfrak{S}^{r_0}(L^2(\mathcal{I}, L^2(\mathbb{R}^d)))} \leq C_0 e^{b_0 |s|} \|W_1\|_{L^{q_0}(\mathcal{I}, \mathcal{W}^{p'_0,p_0}(\mathbb{R}^d))} \|W_2\|_{L^{q_0}(\mathcal{I}, \mathcal{W}^{p'_0,p_0}(\mathbb{R}^d))}
		\]
		
		and
		
		\[ 
		\|W_1T_{a_1+is}W_2\|_{\mathfrak{S}^{r_1}(L^2(\mathcal{I}, L^2(\mathbb{R}^d)))} \leq C_1 e^{b_1 |s|} \|W_1\|_{L^{q_1}(\mathcal{I}, \mathcal{W}^{p'_1,p_1}(\mathbb{R}^d))} \|W_2\|_{L^{q_1}(\mathcal{I}, \mathcal{W}^{p'_1,p_1}(\mathbb{R}^d))}.
		\]
		
		Then, 
		%for any $\theta \in (0,1)$, 
		there exist constants $c$ such that 
		\[\|W_1T_{a+is}W_2\|_{\mathfrak{S}^{r}(L^2(\mathcal{I}, L^2(\mathbb{R}^d)))} \leq C_0^{1-\theta}C_1^{\theta} e^{c |s|} \|W_1\|_{L^{q}(\mathcal{I}, \mathcal{W}^{p',p}(\mathbb{R}^d))} \|(W_2)\|_{L^{q}(\mathcal{I}, \mathcal{W}^{p',p}(\mathbb{R}^d))}
		,\]
		where $a=(1-\theta)a_0+\theta a_1$, and
		%$\frac{1-\theta}{p_0}+\frac{\theta}{p_1}=\frac{1}{p}$, 
		$\frac{1-\theta}{r_0}+\frac{\theta}{r_1}=\frac{1}{r}$. 
		%and $\frac{1-\theta}{q_0}+\frac{\theta}{q_1}=\frac{1}{q}$. 
	\end{corollary}

	In order to establish the Strichartz inequality for orthonormal systems, the following duality principle plays a crucial role.
	\begin{theorem}\label{Duality}
		Let $A:L^2(\mathbb{R}^{d})\to L^q\bigl(\mathcal{I}, \mathcal{W}^{p',p}(\mathbb{R}^d)\bigr)$ be a bounded operator for some $2 \leq q,p \leq \infty$. Consider $\mathcal{I}$ to be some nonempty interval of $\mathbb{R}$. Then for some $\beta \geq 1$, the following are equivalent.
		\begin{enumerate}
			\item There exists a constant $C>0$ such that 
			\[\|WAA^*\overline{W}\|_{\mathfrak{S}^{\beta}(L^2(\mathcal{I}\times \mathbb{R}^d))}\leq C \|{W}\|^2_{L^{\frac{2q}{q-2}}(\mathcal{I}, \mathcal{W}^{\frac{2p}{p+2}, \frac{2p}{p-2}}(\mathbb{R}^d))}.\]\label{WAAstarWinequality}
			\item There exists a constant $C'>0$ such that for any orthonormal system $(f_j)_{j\in J}$ and a sequence $(\alpha_j)_{j\in J}$ of complex numbers,
			\[\Bigg\|\sum_{j\in J}\alpha_j |Af_j|^2\Bigg\|_{L^{\frac{q}{2}}(\mathcal{I}, \mathcal{W}^{\frac{p}{p-2}, \frac{p}{2}}(\mathbb{R}^d))}\leq C'\|\gamma\|_{\mathfrak{S}^{\beta^{'}}(L^2(\mathbb{R}^d))}.\]\label{WAAstarWinequality2}
		\end{enumerate}
	\end{theorem}
	\begin{proof}
		Notice that (\ref{WAAstarWinequality}) is equivalent to the following estimate
		\begin{align}
			\|A^*|W|^2A\|_{\mathfrak{S}^{\beta}(L^2(\mathbb{R}^d))}\leq C \|{W}\|^2_{L^{\frac{2q}{q-2}}(I, \mathcal{W}^{\frac{2p}{p+2}, \frac{2p}{p-2}}(\mathbb{R}^d))}, \label{AstarWAinequality}
		\end{align}
		where, $|W|^2:=\overline{W}W:L^q\bigl(\mathcal{I}, \mathcal{W}^{p',p}(\mathbb{R}^d)\bigr)\to L^{q'}\bigl(\mathcal{I}, \mathcal{W}^{p,p'}(\mathbb{R}^d)\bigr)$. We define an operator 
		$\gamma$ on $L^2(\mathbb{R}^d)$ with braket notation as
		\[\gamma=\sum_{j\in J}\alpha_j \ket{f_j}\bra{f_j}, \]
		where $ f_j\in  L^2(\mathbb{R}^d)$ are eigen functions  of $\gamma$ with corresponding eigen values $\alpha_j$. Using \eqref{AstarWAinequality} and H\"older's inequality in Schatten spaces, we have
		\begin{align}\nonumber
			\operatorname{Tr}_{L^2(\mathcal{I}\times \mathbb{R}^{d})}(WA\gamma (WA)^*)&=\operatorname{Tr}_{L^2(\mathbb{R}^{d})}(\gamma A^*|W|^2A)\\
			&\leq \|\gamma\|_{\mathcal{G}^{\beta '}(L^2(\mathbb{R}^d)} \|A^*|W|^2A\|_{\mathfrak{S}^{\beta}(L^2(\mathbb{R}^d))} \nonumber \\
			&\leq C\|\gamma\|_{\mathfrak{S}^{\beta '}(L^2(\mathbb{R}^d)}\|{W}\|^2_{L^{\frac{2q}{q-2}}(\mathcal{I}, \mathcal{W}^{\frac{2p}{p+2}, \frac{2p}{p-2}}(\mathbb{R}^d))}.\label{inequalityTRW}
		\end{align}
		For $f\in L_t^{q'}\bigl(\mathcal{I}, \mathcal{W}_x^{p,p'}(\mathbb{R}^d)\bigr) $, we have
		\begin{align*}
			|W|^2A\gamma A^*f(t,x)=\int_{\mathcal{I}}\int_{\mathbb{R}^d}\left(\sum_{j\in J}\alpha_jAf_j(t,x)\overline{Af_j(s,y)}|W(t,x)|^2\right)f(s,y)\,dy\,ds.
		\end{align*}
		Therefore, kernel of the operator $|W|^2A\gamma A^*$ is given by 
		\[Ker(|W|^2A\gamma A^*)\left((t,x);(s,y)\right)=\sum_{j\in J}\alpha_jAf_j(t,x)\overline{Af_j(s,y)}|W(t,x)|^2.\]
		Therefore  
		\begin{align}\nonumber
			\operatorname{Tr}_{L^{q'}\bigl(\mathcal{I}, \mathcal{W}^{p,p'}(\mathbb{R}^d)\bigr)}(|W|^2A\gamma A^*)&=\int_{\mathcal{I}}\int_{\mathbb{R}^d} Ker(|W|^2A\gamma A^*)\left((t,x);(t,x)\right)\,dx\,dt\\
			& =\int_{\mathcal{I}}\int_{\mathbb{R}^d}\left(\sum_{j\in J}\alpha_j|Af_j(t,x)|^2\right)|W(t,x)|^2\,dx\,dt.
			\label{equalityTRkern}
		\end{align}
		Using the fact that $\operatorname{Tr}_{L^{q'}\bigl(\mathcal{I}, \mathcal{W}^{p,p'}(\mathbb{R}^d)\bigr)}(|W|^2A\gamma A^*)=\operatorname{Tr}_{L^2(\mathcal{I}\times \mathbb{R}^{d})}(WA\gamma (WA)^*)$ together with \eqref{inequalityTRW} and \eqref{equalityTRkern}, we get
		\begin{align*}
			\int_{\mathcal{I}}\int_{\mathbb{R}^d}\left(\sum_{j\in J}\alpha_j|Af_j(t,x)|^2\right)|W(t,x)|^2\,dx\,dt\leq  C\|\gamma\|_{\mathfrak{S}^{\beta '}(L^2(\mathbb{R}^d))}\|{W}\|^2_{L^{\frac{2q}{q-2}}\Bigl(I, \mathcal{W}^{\frac{2p}{p+2}, \frac{2p}{p-2}}(\mathbb{R}^d)\Bigr)}. 
		\end{align*}
		Letting $V=|W|^2$, we obtain
		\begin{align*}
			\int_{\mathcal{I}}\int_{\mathbb{R}^d}\left(\sum_{j\in J}\alpha_j|Af_j(t,x)|^2\right)V(t,x)\,dx\,dt\leq  C\|\gamma\|_{\mathfrak{S}^{\beta '}(L^2(\mathbb{R}^d))}\|V\|_{L^{\frac{q}{q-2}}\Bigl(\mathcal{I}, \mathcal{W}^{\frac{p}{2}, \frac{p}{p-2}}(\mathbb{R}^d)\Bigr)}
		\end{align*}
		and  by  the duality principle, we have the estimate (\ref{WAAstarWinequality2}).  
	\end{proof}
	In a similar fashion, we also can prove the following  duality principle. 
	\begin{theorem}\label{DualityWAspace}
		Let $A:L^2(\mathbb{R}^{d})\to \mathcal{W}_t(L^{q_1},L^{q_2}) \bigl(\mathcal{I}, \mathcal{W}_x^{p',p}(\mathbb{R}^d)\bigr)$ to be a bounded operator for some $2 \leq q_1,q_2,p \leq \infty$. Consider $\mathcal{I}$ be some nonempty interval of $\mathbb{R}$. Then for some $\beta \geq 1,$ the following are equivalent.
		\begin{enumerate}
			\item There exists a constant $C>0$ such that 
			\[\|WAA^*\overline{W}\|_{\mathfrak{S}^{\beta}(L^2(\mathcal{I}\times \mathbb{R}^d))}\leq C \|{W}\|^2_{\mathcal{W}_t(L^{\frac{2q_1}{q_1-2}},L^{\frac{2q_2}{q_2-2}})(\mathcal{I}, \mathcal{W}^{\frac{2p}{p+2}, \frac{2p}{p-2}}(\mathbb{R}^d))}.\]%\label{WAAstarWinequality}
			\item There exists a constant $C'>0$ such that for any orthonormal system $(f_j)_{j\in J}$ and a sequence $(\alpha_j)_{j\in J}$ of complex numbers,
			\[\Bigg\|\sum_{j\in J}\alpha_j |Af_j|^2\Bigg\|_{\mathcal{W}_t(L^{\frac{q_1}{2}},L^{\frac{q_2}{2}})(\mathcal{I}, \mathcal{W}^{\frac{p}{p-2}, \frac{p}{2}}(\mathbb{R}^d))}\leq C'\|\gamma\|_{\mathfrak{S}^{\beta^{'}}(L^2(\mathbb{R}^d))}.\]
		\end{enumerate}
		
	\end{theorem}
	The proof of Theorem \ref{DualityWAspace} proceeds analogously to that of Theorem \ref{Duality}. We note that Theorem \ref{DualityWAspace} is, in fact, more general than Theorem \ref{Duality}.
	
	\section{Orthonormal  Strichartz estimates  and end point case for harmonic oscillator \texorpdfstring{$\mathcal{H}^{+}$}{}}\label{ONSH+}
	In this section, we shall establish the improved single-function Strichartz estimate, together with its orthonormal counterpart, for the Harmonic oscillator \(\mathcal{H}^{+}\). We begin with Theorem \ref{Wttheorem1}, which will play a pivotal role in establishing both the improved Strichartz estimate and its orthonormal analogue.
	For a fixed $t\in \mathcal{I}:=(\frac{-\pi}{4},\frac{\pi}{4})$, we know that $\|e^{it\mathcal{H}^{+}}f\|_{L_x^2}=\|f\|_{L_x^2}.$ Therefore 
	\begin{align*}
		\|e^{it\mathcal{H}^{+}}f\|_{L^2_tL_x^2}=\frac{\pi}{2}\|f\|_{L_x^2}.
	\end{align*}
	Suppose $A=e^{it\mathcal{H}^{+}}$. Then, we define $T=AA^*:L^{2}\bigl(\mathcal{I}, \mathcal{W}^{2,2}(\mathbb{R}^d)\bigr)\to L^{2}\bigl(\mathcal{I}, \mathcal{W}^{2,2}(\mathbb{R}^d)\bigr)$ as follows
	\begin{align}\nonumber
		TF(t,x)&:=\int_{\mathcal{I}}e^{i(t-s)\mathcal{H}^+}F(s,x)\, ds\\
		&=\int_{\mathcal{I}}\int_{\mathbb{R}^d}K_{i(t-s)}(x,y)F(s,y)\,dy\,ds,\label{defTwrtH+}
	\end{align}
	where $K_{it}$ is defined by \eqref{CH2k}. Now we have the following result in a general framework with regard to the time decay of the kernel $K_{it}$.

	\begin{theorem}\label{Wttheorem1}
		Let $T$ be an operator as in \eqref{defTwrtH+}. The kernel has following uniform decay estimate \begin{align*}
			|K_{it}(x,y)|\lesssim \frac{1}{|t|^{\lambda}} \ \ \ \textit{for some}  \ \ \lambda> 0.
		\end{align*}
		Then for any $ 2\leq p,q \leq \infty $ such that 
		\[
		\frac{1}{q} + \frac{\lambda}{p} = \frac{1}{2} \quad \text{and} \quad 2\lambda+1 < p \leq 2(\lambda+1),
		\]
		we have
		\begin{align}\label{dual inequality}
			\| W_1 T W_2 \|_{\mathfrak{S}^p(L^2(\mathcal{I}, L^2(\mathbb{R}^d)))} \leq C \| W_1 \|_{L^{q}(\mathcal{I},  \mathcal{W}^{p',p}(\mathbb{R}^d))} \| W_2 \|_{L^{q}(\mathcal{I}, \mathcal{W}^{p',p}(\mathbb{R}^d))},
		\end{align}
		where $W_1$ and $W_2$ should be interpreted as multiplication operator corresponding to the function $W_1$ and $W_2$, respectively.
	\end{theorem}
	
	\begin{proof}
		For $\epsilon>0$, we define 
		\[T_{\epsilon}F(t,x)=\int_{\mathcal{I}}\int_{\mathbb{R}^d}K_{\epsilon}(x,y,t-s)F(s,y)\,dy\,ds,\]
		where
		\[K_{\epsilon}(x,y,t)=\chi_{\epsilon<|t|< \frac{\pi}{4}}(t)K_{it}(x,y).\]
		To use Stein's analytic complex interpolation, for $z \in \mathbb{C}$ with $\operatorname{Re}(z) \in [-1, \lambda]$, we define
		\[
		K_{z,\varepsilon}(x, y, t) = t^{z} K_{\varepsilon}(x, y, t)
		\]
		and define the corresponding operator as
		\[
		T_{z,\varepsilon} F(t,x) = \int_{\mathcal{I}} \int_{\mathbb{R}^d}  K_{z,\varepsilon}(x, y, t - s) F(s,y)  \, dy \, ds.
		\]
		Then for any $\varepsilon > 0$, $\{T_{z,\varepsilon}\}_z$ forms an analytic family of operators in the sense of Stein. Moreover, we can deduce that
		\[
		|K_{z,\varepsilon}(x, y, t)| \lesssim \sup_{x, y \in \mathbb{R}^d}|t|^{\operatorname{Re}(z) - \lambda}, \quad \forall t \in \mathcal{I}. 
		\]
		Using the Hardy-Littlewood-Sobolev inequality in Lebesgue spaces, we obtain
		\begin{align}\nonumber
			\|W_1 T_{z,\varepsilon} W_2 \|_{\mathfrak{S}^2(L^2(\mathcal{I}, L^2(\mathbb{R}^d)))}^2 &= \int_{\mathcal{I} \times \mathcal{I}} \int_{\mathbb{R}^d \times \mathbb{R}^d} |W_1(t,x)|^2 |K_{z,\varepsilon}(x, y, t-s)|^2 |W_2(s,y)|^2 \, dx \, dy \, dt \, ds\\\nonumber
			& \lesssim \int_{\mathcal{I} \times \mathcal{I}} \int_{\mathbb{R}^d \times \mathbb{R}^d} \frac{|W_1(t,x)|^2 |W_2(s,y)|^2}{|t-s|^{2\lambda - 2\operatorname{Re}(z)}} \, dx \,  dy \, dt \, ds\\\nonumber
			&= \int_{\mathcal{I} \times \mathcal{I}} \frac{\| W_1(t,\cdot) \|_{L^2(\mathbb{R}^d)}^2 \| W_2(s,\cdot) \|_{L^2(\mathbb{R}^d)}^2}{|t-s|^{2 \lambda - 2\operatorname{Re}(z)}} \, dt \, ds\\
			&\lesssim \| W_1 \|_{L^{2\tilde{u}}(\mathcal{I}, L^2(\mathbb{R}^d))}^2 \| W_2 \|_{L^{2\tilde{u}}(\mathcal{I}, L^2(\mathbb{R}^d))}^2,\label{2norminqualityofWTW}
		\end{align}
		provided that $0\leq 2\lambda-2\operatorname{Re}(z)<1$ and $\frac{2}{\tilde{u}}+2\lambda-2\operatorname{Re}(z)=2$. If we write $u=2\tilde{u}$, then $\frac{1}{u}\in (\frac{1}{4},\frac{1}{2}]$ and
		and
		\begin{equation}\label{Interpolation1H+}
			\big\|W_1 T_{z,\varepsilon} W_2\big\|_{\mathfrak{S}^2(L^2(\mathcal{I}, L^2(\mathbb{R}^d)))}\lesssim \|W_1\|_{L^u\left(\mathcal{I},L^2(\mathbb{R}^d)\right)}\|W_2\|_{L^u\left(\mathcal{I},L^2(\mathbb{R}^d)\right)},
		\end{equation}
		provided that $\frac{1}{u}=\frac{1}{2}-\frac{1}{2}(\lambda-\operatorname{Re}(z))$ and $\operatorname{Re}(z)\in (\frac{2\lambda-1}{2},\lambda]$. Now, our aim is to show that for $\operatorname{Re}(z)=-1$, the operator 
		\begin{align*}
			T_{z,\varepsilon}: L^{2}\bigl(\mathcal{I}, \mathcal{W}^{2,2}(\mathbb{R}^d)\bigr) \to L^{2}\bigl(\mathcal{I},\mathcal{W}^{2,2}(\mathbb{R}^d)\bigr)
		\end{align*}
		is bounded with some constant that depends only on $d$ and $|\operatorname{Im}(z)|$ exponentially. Notice that 
		\begin{align*}
			T_{z,\varepsilon} F(t,x) &= \int_{\mathcal{I}}\int_{\mathbb{R}^d}  K_{z,\varepsilon}(x, y, t-s) F(s,y)  \, dy\, ds\\
			&= \int_{\varepsilon < |s| < \pi/4} s^{-1 + i \operatorname{Im}(z)}  \left(\int_{\mathbb{R}^d} K_{i s}(x, y) F(t-s,y) \, dy\right) \, ds\\
			&= \int_{\varepsilon < |s| < \pi/4 }s^{-1 + i \operatorname{Im}(z)} e^{i s \mathcal{H}^+} F(t-s,x) \, ds.
		\end{align*}
		For fixed $z$ and $t$, we can write 
		\begin{align*}
			\left(e^{-i t \mathcal{H}^+}T_{z,\varepsilon} F\right)(t,x)& =\int_{\mathbb{R}^d}K_{-it}(x,y)T_{z,\varepsilon} F(t,y)\, dy\\
			&= \int_{\mathbb{R}^d}K_{-it}(x,y)\left(\int_{\mathbb{R}^d} \int_{\mathcal{I}} K_{z,\varepsilon}(y, w, s) F(t-s,w) \, ds \, dw\right)\, dy \\
			&=\int_{\varepsilon < |s| < \pi/4} s^{-1 + i \operatorname{Im}(z)} \left(\int_{\mathbb{R}^d}K_{-it}(x,y) e^{i s \mathcal{H}^+} F(t-s,y)\, dy\right) \, ds \\
			&=\int_{\varepsilon < |s| < \pi/4} s^{-1 + i \operatorname{Im}(z)} e^{-i (t-s) \mathcal{H}^+} F(t-s,x) \, ds\\
			&=: \int_{\varepsilon < |s| < \pi/4} s^{-1 + i \operatorname{Im}(z)} G(t-s,x) \, ds,
		\end{align*}
		where $G(t,x) = e^{-i t \mathcal{H}^+} F(t,x)$. Therefore
		\begin{align}
			\|e^{-i t \mathcal{H}^+}T_{z,\varepsilon} F\|_{\mathcal{W}^{2,2}(\mathbb{R}^d)}
			= \left\| \int_{\varepsilon < |s| < \pi/4} s^{-1 + i \operatorname{Im}(z)} G(t-s,\cdot) \, ds \right\|_{\mathcal{W}^{2,2}(\mathbb{R}^d)},\label{normequalityeitHTF} 
		\end{align}
		and since the operator $e^{-i t \mathcal{H}^+}:L^2(\mathbb{R}^d)\to L^{2}(\mathbb{R}^d)$ is unitary for fixed $t$,
		it follows that 
		\begin{align}\label{GFequality23}
			\| G(t,\cdot) \|_{\mathcal{W}^{2,2}(\mathbb{R}^d)} = \| F(t, \cdot) \|_{L^2(\mathbb{R}^d)}.
		\end{align}
		If we define
		\[
		\mathcal{H}_{z,\varepsilon} : h(t) \mapsto \int_{\varepsilon < |s| <\pi/4} s^{-1 + i \operatorname{Im}(z)} h(t-s) \, ds,
		\]
		then using \eqref{normequalityeitHTF}, we get
		
		\[\|e^{-i t \mathcal{H}^+}T_{z,\varepsilon} F\|_{\mathcal{W}^{2,2}(\mathbb{R}^d)}=\|\mathcal{H}_{z,\varepsilon}G\|_{\mathcal{W}^{2,2}(\mathbb{R}^d)}.\]
		This eventually gives  
		\begin{align}
			\| {T}_{z,\varepsilon} F \|_{L^2(\mathcal{I}, \mathcal{W}^{2,2}(\mathbb{R}^d))} = \|\mathcal{H}_{z,\varepsilon}G\|_{L^2(\mathcal{I},{\mathcal{W}^{2,2}(\mathbb{R}^d))}}=\| \mathcal{H}_{z,\varepsilon} G \|_{L^2(\mathcal{I}, L^2(\mathbb{R}^d))} = \| \mathcal{H}_{z,\varepsilon} G \|_{L^2(\mathbb{R}^d, L^2(\mathcal{I}))}. \label{normequalityTH}
		\end{align}
		Since the operator $\mathcal{H}_{z,\varepsilon}$ is just Hilbert transform up to $i \operatorname{Im}(z)$, we have the bound $\mathcal{H}_{z,\varepsilon} : L^2(\mathcal{I}) \to L^2(\mathcal{I})$ with some constant depending only on $\operatorname{Im}(z)$ exponentially. Now combining \eqref{GFequality23} and \eqref{normequalityTH}, we obtain
		\begin{align*}
			\| {T}_{z,\varepsilon} F \|_{L^2(\mathcal{I}, \mathcal{W}^{2,2}(\mathbb{R}^d))} \leq C(\operatorname{Im}(z)) \| G \|_{L^2(\mathbb{R}^d, L^2(\mathcal{I}))}
			= C(\operatorname{Im}(z)) \| G \|_{L^2(\mathcal{I}, L^2(\mathbb{R}^d))}
			=C(\operatorname{Im}(z)) \| F \|_{L^2(\mathcal{I}, \mathcal{W}^{2,2}(\mathbb{R}^d))},
		\end{align*}
		which is our desired bound for ${T}_{z,\varepsilon} : L^2(\mathcal{I}, \mathcal{W}^{2,2}(\mathbb{R}^d)) \to L^2(\mathcal{I}, \mathcal{W}^{2,2}(\mathbb{R}^d))$ when $\operatorname{Re}(z) = -1$.
		%{\color{red}
			Noting the fact that $\mathfrak{S}^\infty$-norm is the usual operator norm, for $\operatorname{Re}(z) = -1$, we get
			\begin{align}\nonumber
				\| W_1 {T}_{z,\varepsilon} W_2 \|_{\mathfrak{S}^\infty(L^2(\mathcal{I}\times \mathbb{R}^d))}&=\| W_1 {T}_{z,\varepsilon} W_2 \|_{L^2(\mathcal{I}\times \mathbb{R}^d)\to L^2(\mathcal{I}\times \mathbb{R}^d)} \\ 
				%&\leq C(\operatorname{Im}(z)) \| W_1\| \| W_2\|\\ 
				&\leq C(\operatorname{Im}(z)) \| W_1\|_{L^{\infty}(\mathcal{I}, \mathcal{W}^{1, \infty}(\mathbb{R}^d))} \| W_2\|_{L^{\infty}(\mathcal{I}, \mathcal{W}^{1, \infty}(\mathbb{R}^d))}.\label{infinitynormofWTWH+}
			\end{align}
			In the last inequality, we have used the fact that $\|Wg\|_{L^2(\mathcal{W}^{2,2})}\leq \| W\|_{L^{\infty}(\mathcal{I}, \mathcal{W}^{1, \infty}(\mathbb{R}^d))}\|g\|_{L^2(\mathcal{W}^{2,2})}$ for every $g\in L^2(\mathcal{W}^{2,2})$ [see Proposition \ref{algebraproperty}].
			Now applying the complex interpolation result between \eqref{infinitynormofWTWH+} and \eqref{Interpolation1H+} on Wiener amalgam spaces, established in Corollary \ref{interpolationtheorem}, for $z=0$,  we get 
			\begin{equation}\label{Interpolation3}
				\| W_1 T_\varepsilon W_2 \|_{\mathfrak{S}^p(L^2(\mathcal{I}, L^2(\mathbb{R}^d)))} \leq C \| W_1 \|_{L^{q}(\mathcal{I},  \mathcal{W}^{p',p}(\mathbb{R}^d))} \| W_2 \|_{L^{q}(\mathcal{I}, \mathcal{W}^{p',p}(\mathbb{R}^d))},
			\end{equation} 
			for some $C>0$ independent of $\varepsilon$ as long as
			\begin{equation*}
				\frac{1}{q}+\frac{\lambda}{p}=\frac{1}{2},\; \text{with}\; 2\lambda+1<p\leq 2(\lambda+1).
			\end{equation*}
			Finally,     letting $\varepsilon\rightarrow 0^+$ in (\ref{Interpolation3}), we complete the proof of the theorem. 
		\end{proof}

		\subsection{Improvements of classical Strichartz estimates for \texorpdfstring{$\mathcal{H}^+$}{}}
		Note that the Strichartz estimate in \eqref{CorNicostrichartz} is somewhat weak, since the exponent in the time variable is only $\frac{q}{2}$. One of our main observations is that this exponent can be improved from $\frac{q}{2}$ to $q$. The following result gives the corresponding improved version of \eqref{CorNicostrichartz}.
		\begin{proof}[\bf{Proof of Theorem \ref{strichartzH+improved}}]
			Since 
			\begin{align*}
				\|WAA^*\overline{W}\|_{\mathfrak{G}^{\frac{2p}{p-2}}(L^2(\mathcal{I}\times \mathbb{R}^d)\to L^2(\mathcal{I}\times \mathbb{R}^d))}=\|A^*|W|^2A\|_{\mathfrak{G}^{\frac{2p}{p-2}}(L^2( \mathbb{R}^d)\to L^2(\mathbb{R}^d))},
			\end{align*}
			we use Theorem \ref{Wttheorem1} for $\lambda=d/2$ to obtain the following inequality \begin{align}\label{strichartzWAA*Winequality}
				\|A^*|W|^2A\|_{\mathfrak{G}^{\frac{2p}{p-2}}(L^2( \mathbb{R}^d)\to L^2(\mathbb{R}^d))}\lesssim \|W\|_{L_t^{\frac{2q}{q-2}}\mathcal{W}_x^{\frac{2p}{p+2},\frac{2p}{p-2}}}^2, 
			\end{align}
			for $\frac{2}{q}+\frac{d}{p}=\frac{d}{2}$ and $\frac{2(d+2)}{d}\leq \frac{p}{2} < \frac{d+1}{d-1}$. Let us  take $f\in L^2(\mathbb{R}^d)$ and make use of inequality \eqref{strichartzWAA*Winequality} together with the fact $\|G\|_{\operatorname{op}} \lesssim \|G\|_{\mathfrak{S}^{\alpha}}$ for any $\alpha \geq 1$. Here $\|G\|_{\operatorname{op}}$ denotes the operator norm of operator $G$. Then 
			\begin{align*}
				\|WAf\|^2_{L^2(\mathcal{I}\times \mathbb{R}^d)}&=\braket{WAf,WAf}\\
				&=\braket{A^*|W|^2Af,f}\\
				& \leq \|A^*|W|^2A\|_{\operatorname{op}} \|f\|^2_{L^2}\\
				& \leq \|A^*|W|^2A\|_{\mathfrak{G}^{\frac{2p}{p-2}}(L^2( \mathbb{R}^d)\to L^2(\mathbb{R}^d))} \|f\|^2_{L^2}\\
				& \leq \|W\|_{L_t^{\frac{2q}{q-2}}\mathcal{W}_x^{\frac{2p}{p+2},\frac{2p}{p-2}}}^2\|f\|^2_{L^2}.
			\end{align*}
			Now suppose $|W|^2=V$, then 
			\begin{align*}
				\int_{\mathcal{I}\times \mathbb{R}^d}|Af|^2(t,x)V(t,x)\, dt\, dx \lesssim \|V\|_{L_t^{\frac{q}{q-2}}\mathcal{W}_x^{(\frac{p}{p-2})',\frac{p}{p-2}}}\|f\|^2_{L^2}.
			\end{align*}
			Therefore $|Af|^2\in L_t^{(\frac{q}{2})}\mathcal{W}_x^{(\frac{p}{2})',\frac{p}{2}},$
			which in turn implies that 
			\begin{align}\label{strichartzH+halfrange}
				\|e^{it\mathcal{H}^{+}} f\|_{L^{q}_t W^{p',p}_x} \lesssim \|f\|_{L^2_x},
			\end{align}
			for $\frac{2}{q}+\frac{d}{p}=\frac{d}{2}$ and $\frac{4(d+2)}{d}\leq  p < \frac{2(d+1)}{d-1}$. We also have the energy estimate
			\begin{align}\label{strichartzH+trivial}
				\|e^{it\mathcal{H}^{+}} f\|_{L^{\infty}_t W^{2,2}_x} =\|f\|_{L^2_x},
			\end{align}
			Interpolating  between the estimates \eqref{strichartzH+halfrange} and \eqref{strichartzH+trivial},  we have our required estimate.
		\end{proof}

		% \begin{proof}
			%     To prove inequality \eqref{endpointsingstrichartz}, it is equivalent to show that 
			%     \begin{align}
				%         \Bigg|\int \int \langle e^{-is\mathcal{H}_{\mathcal{A}}}F(s),  e^{-it\mathcal{H}_{\mathcal{A}}}G(t)\rangle \, ds \, dt\Bigg|\lesssim\|F\|_{L_{t}^2(W_x^{r',r})}\|G\|_{L_{t}^2(W_x^{r',r})}
				%     \end{align}
			% \end{proof}

		\begin{remark}
			Theorem \ref{strichartzH+improved} improves the Strichartz estimates \eqref{CorNicostrichartz} for the harmonic oscillator by proving boundedness in the time-space norm $L^{q}_t \mathcal{W}^{p',p}_x$ instead of their $L^{q/2}_t \mathcal{W}^{p',p}_x$ while maintaining the same admissible condition. Since $L^{q}\subset L^{q/2}$ for $q>4$ on finite intervals, our estimate places the solution in a strictly smaller (more restrictive) time space, thereby providing stronger temporal regularity without sacrificing the refined spatial information offered by the Wiener amalgam space.   However, it is important to note that while our estimate is stronger, it is only valid for a restricted range of admissible exponents $p$.  Thus, our result offers a trade-off: a stronger time-space bound at the expense of a restricted range of spatial regularity parameters.

			%We also have from \cite[Theorem 5.2]{NFC}
			
			%\begin{equation} \label{E6.5}
			%	\|e^{it\mathcal{H}^{+}} f\|_{L^{q/2}_t W^{p',p}_x} \lesssim \|f\|_{L^2_x}.
			%\end{equation}
			%Interpolation between \eqref{strichartzforH^+} and \eqref{E6.5}, one can obtain improved Strichartz estimates, here we omit the details.
			
		\end{remark}

		\begin{proof}[Proof of the theorem \ref{ONSH+improved}]\label{proofONSH+}
			From the fact that the operator $e^{it\mathcal{H}^{+}}$ is unitary on $L^2(\mathbb{R}^d)$, it follows that
			\begin{equation*}
				\left\|\sum_{j\in J}\alpha_j|e^{it\mathcal{H}^{+}}f_j|^2\right\|_{L^\infty(\mathcal{I},\mathcal{W}^{\infty,1}(\mathbb{R}^d))} \leq \sum_{j\in J}|\alpha_j|,
			\end{equation*}
			for any (possibly infinite) orthonormal system $\{f_j\}_{j\in J}$ in $L^2(\mathbb{R}^d)$ and any sequence $\{\alpha_j\}_{j\in J}$ in $\mathbb{C}$. That means equation (\ref{ONShermite1}) holds for $(q, p)=(\infty, 1).$   
			On the other hand,  from Theorem \ref{Wttheorem1},    the duality argument in Theorem  \ref{Duality} yields the  orthonormal inequality  (\ref{ONShermite1}) that holds for  $1+\frac{2}{d} \leq p < \frac{d+1}{d-1}$.  Combining these two ranges for $p$, we have the required estimate  (\ref{ONShermite1}) for $1 \leq p < \frac{d+1}{d-1}$.
		\end{proof}
		From Lemma \ref{Lemmaratnakumar}, we know that the kernel of $AA^*$ exhibits uniform decay with  $\lambda=\frac{d}{2}$, i.e., 
		$|K_{it}(x,y)|\lesssim \frac{1}{|t|^{\frac{d}{2}}}.$
		However, for small time $t\in \mathcal{I}$,  we have
		$ |K_{it}(x,y)|\lesssim \frac{1}{|t|^{\frac{d}{2}}}\leq   \frac{1}{|t|^{d}}
		$ for all $t\in \mathcal{I}.$ Thus, by considering $\lambda=d$,  we obtain   Theorem \ref{ONSq/2}.

		\subsection{Optimality of Schatten exponent for \texorpdfstring{$\mathcal{H}^+$}{}} 		
		The inequality \eqref{ONShermite1} can be reformulated in a more convenient manner by introducing the operator
		\begin{align}\label{defgamma}
			\gamma := \sum_{j} {\alpha}_j \ket{f_j}\bra{f_j},
		\end{align}
		which acts on \(L^2(\mathbb{R}^d)\). The time-evolved operator is given by
		\[
		\gamma(t) := e^{it\mathcal{H}^+} \gamma e^{-it\mathcal{H}^+}
		= \sum_{j} {\alpha}_j \ket{e^{it\mathcal{H}^+} f_j} \bra{e^{it\mathcal{H}^+} f_j}.
		\]
		Introducing the density function
		\[
		\rho_{\gamma(t)} := \sum_{j} {\alpha}_j |e^{it\mathcal{H}^+} f_j|^2,
		\]
		we observe that \eqref{ONShermite1} can be equivalently expressed as
		\begin{equation}\label{equivalentONS}
			\|{ \rho_{\gamma(t)} }\|_{L^q_t(\mathbb{R}, \mathcal{W}_x^{p',p}(\mathbb{R}^d))}
			\le C_{d,q} \|{ \gamma }\|_{\mathfrak{S}^{\frac{2p}{p+1}}}, 
		\end{equation}
		where
		\[
		\|{ \gamma }\|_{\mathfrak{S}^{\frac{2p}{p+1}}}
		:= \left( \sum_{j} |{\alpha}_j|^{\frac{2p}{p+1}} \right)^{\frac{p+1}{2p}},
		\]
		denotes the Schatten norm of the operator \(\gamma\).
		A significant advantage of the formulation \eqref{equivalentONS} is that it obviates the need to explicitly specify the functions \(f_j\) and the coefficients \({\alpha}_j\); all the information is now encoded within the operator \(\gamma\) itself. This new formation will be now useful to prove the optimality of the exponent $\frac{2p}{p+1}$ in the right side of the inequality \eqref{ONShermite1}.  Regarding this we have the following result. 
		
		\begin{theorem}[Optimality of the Schatten norm]\label{optimalityschattenexponent}
			Assume that $d, p, q \geq 1$ satisfy
			\begin{equation}
				\frac{2}{p} + \frac{d}{q} = d \qquad \textit{and} \quad 1\leq p < \frac{d+1}{d-1}.
			\end{equation}
			Then, for every $r > \frac{2p}{p+1}$, we have
			\begin{equation}
				\sup_{\gamma \in \mathfrak{S}^r}
				\frac{\left\| \rho_{e^{it\mathcal{H}^+}\gamma e^{-it\mathcal{H}^+}} \right\|_{L^{q}(\mathcal{I},\mathcal{W}^{p',p}(\mathbb{R}^d))}}{\|\gamma\|_{\mathfrak{S}^r}}
				= +\infty. \label{optimalityeqution}
			\end{equation}
		\end{theorem}
		\begin{proof} 
			Our proof strategy is very much similar to the proof of \cite[Proposition 1]{frank}. It involves constructing a family of operators $\gamma$ that depends on three positive parameters $\beta$, $L$, and $\mu$; these will be fixed suitably toward the end of the argument. To build $\gamma$, we make use of coherent states $F_{x,\xi}$, indexed by $x, \xi \in \mathbb{R}^d$, which are defined as
			\begin{equation}
				F_{x,\xi}(z) = (2\pi\beta)^{-d/4} e^{-|z-x|^2/(4\beta)} e^{i\xi \cdot z}.
			\end{equation}
			These functions are normalized in $L^2(\mathbb{R}^d)$ and satisfy the resolution of the identity
			\begin{equation}
				\int_{\mathbb{R}^d \times \mathbb{R}^d} \frac{dx \, d\xi}{(2\pi)^d} |F_{x,\xi}\rangle \langle F_{x,\xi}| = I.
			\end{equation}
			With these ingredients, we set
			\begin{equation}
				\gamma = \int_{\mathbb{R}^d \times \mathbb{R}^d} \frac{dx \, d\xi}{(2\pi)^d} e^{-|x|^2/L^2 - |\xi|^2/\mu} |F_{x,\xi}\rangle \langle F_{x,\xi}|.
			\end{equation}
			A key quantity in our subsequent analysis is
			\begin{equation}
				N = \int_{\mathbb{R}^d} \gamma(z,z) \, dz
				= \int_{\mathbb{R}^d \times \mathbb{R}^d} \frac{dx \, d\xi}{(2\pi)^d} e^{-|x|^2/L^2 - |\xi|^2/\mu}
				= A_d L^d \mu^{d/2},
			\end{equation}
			where the constant $A_d$ is explicit and depends only on the dimension $d$.
			
			Clearly, $\gamma$ is a nonnegative operator. Furthermore, an application of the Berezin-Lieb inequality \cite{berezin,lieb} yields
			\begin{equation}
				\operatorname{Tr} \gamma^r \leq \int_{\mathbb{R}^d \times \mathbb{R}^d} \frac{dx \, d\xi}{(2\pi)^d} e^{-r |x|^2/L^2 - r |\xi|^2/\mu}
				= r^{-d} N
			\end{equation}
			for all $r \geq 1$.
			Consequently, the denominator appearing in \eqref{optimalityeqution} is bounded above by $r^{-d/r} N^{1/r}$, uniformly with respect to $\beta$.
			
			To complete the proof, we now demonstrate that an appropriate choice of the parameters $\beta$, $L$, and $\mu$ allows us to control the numerator from below by a constant multiple of $N^{(p+1)/(2p)}$. Subsequently, we take $\mu$ and $L$ to be large, which yields the desired result in the limiting regime $N \to \infty$.
			
			In order to implement this plan, we compute the left-hand side explicitly. Below we outline the essential steps of the calculation. To begin with, we have
			\begin{equation}
				e^{it\mathcal{H}^{+}} F_{x,\xi}(z)
				= (-2\pi i \sin 2t)^{-d/2} (2\pi \beta)^{-n/4}
				\int_{\mathbb{R}^d}
				e^{\frac{i}{2} \cot 2t (|z|^2 + |y|^2) + \frac{i}{\sin(2t)} z \cdot y}
				e^{-\frac{|y-x|^2}{4\beta}} e^{i \xi \cdot y} \, dy.
			\end{equation}
			This implies that \[
			\left| e^{it\mathcal{H}^{+}} F_{x,\xi}(z) \right|
			=
			\left( \frac{2\beta}{\pi (4\beta^2 \cos^2 2t + \sin^2 2t)} \right)^{d/4}
			e^{-
				\frac{\beta |z - x \cos 2t + \xi \sin 2t|^2}{4\beta^2 \cos^2 2t + \sin^2 2t}
			}.
			\]

			Finally, the density function $\rho_{\gamma(t)}(z) := \rho_{e^{it\mathcal{H}_{\mathcal{A}}}\gamma e^{-it\mathcal{H}_{\mathcal{A}}}}(z)$ is given by
			\begin{align}\nonumber
				\rho_{\gamma(t)}(z)
				&= \int_{\mathbb{R}^d \times \mathbb{R}^d} \frac{dx \, d\xi}{(2\pi)^d} e^{-|x|^2/L^2 - |\xi|^2/\mu} |e^{it\mathcal{H}^{+}} F_{x,\xi}(z)|^2\\
				&= \left( \frac{2\pi \beta\mu L^2}{(4\beta^2+2\beta L^2) \cos^2 2t + (1 + 2\mu \beta) \sin^2 2t} \right)^{d/2}
				e^{-
					\frac{2\beta |z|^2}{(4\beta^2+2\beta L^2) \cos^2 2t + (1 + 2\mu \beta) \sin^2 2t}
				}\\
				&=:A(t)^{d/2}e^{-B(t)|z|^2},
			\end{align}
			where $A(t)=\frac{2\pi \beta\mu L^2}{(4\beta^2+2\beta L^2) \cos^2 2t + (1 + 2\mu \beta) \sin^2 2t} $ and $B(t)=\frac{2\beta}{(4\beta^2+2\beta L^2) \cos^2 2t + (1 + 2\mu \beta) \sin^2 2t}$, and consequently $\frac{A(t)}{B(t)}=\pi \mu L^2$.

            Our next aim is to calculate $L^{q}\mathcal{W}^{p',p}$ norm of $\rho_{\gamma(t)}(z)$. We note that 
			\begin{align}\nonumber
				\check{\rho}_{\gamma(t)}(\xi)=\left(\frac{\mu L^2}{4}\right)^{\frac{d}{2}}e^{-\pi |\Lambda\xi|^2},
			\end{align}
			where $\Lambda$ is a scalar matrix with entries to be equal to $(\frac{1}{4B(t)\pi})^{1/2}$. Using \cite[Lemma 3.2]{NFC},  we have
			\begin{align*}
				\|\rho_{\gamma(t)}\|_{\mathcal{W}^{p',p}}\|=\|\check{\rho}_{\gamma(t)}\|_{M^{p',p}}\sim (\mu L^2)^{\frac{d}{2}}B(t)^{\frac{d}{2}\left(\frac{p-1}{p}\right)}.
			\end{align*}
			Now taking $L_{t}^q$ norm of $ \|\rho_{\gamma(t)}\|_{\mathcal{W}^{p',p}}$ with respect to time variable, we obtain
			\begin{align*}
				\int_{\mathcal{I}} \|\rho_{\gamma(t)}\|^{q}_{\mathcal{W}^{p',p}} \,dt &=(\mu L^2)^{\frac{dq}{2}}\int_{\mathcal{I}} B(t)^{\frac{dq}{2}\left(\frac{p-1}{p}\right)}\, dt.
			\end{align*}
			Using  the admissible condition  $\frac{dq}{2}\left(\frac{p-1}{p}\right)=1$ and the fact that  $$ \int_{\mathcal{I}}B(t)\, dt=\frac{\beta}{\sqrt{\beta L^2+2 \beta^2}\sqrt{2+4\beta \mu}}, $$ we have
			\begin{align*}
				\|\rho_{\gamma(t)}\|_{L^q(\mathcal{I},\mathcal{W}^{p',p})}=(\mu L^2)^{\frac{d}{2}}\left(\frac{\beta}{\sqrt{\beta L^2+2 \beta^2}\sqrt{2+4\beta \mu}}\right)^{\frac{1}{q}}.
			\end{align*}
			The remaining calculation can be carried out by following the same arguments as those used in the proof  \cite[Proposition 1]{frank}.

		\end{proof}

		\subsection{Endpoint estimates for \texorpdfstring{$\mathcal{H}^+$}{\mathscr{G}(n,p)}}
		We claim that  at the end point $p=\frac{d+1}{d-1}$, the ONS estimate in Theorem  \ref{ONSH+improved} does not hold. To prove this end point case we need to form a dual version of the estimate \eqref{ONShermite1}. We recall that for any (locally) trace-class operator \(\gamma\) and any bounded function \(V\) of compact support,
		\[
		\operatorname{Tr}\big( V(x) \gamma \big) = \int_{\mathbb{R}^d} V(x) \rho_\gamma(x) \, dx,
		\]
		where \(V(x)\) on the left is identified with the corresponding multiplication operator on \(L^2(\mathbb{R}^d)\). For a time-dependent potential \(V(t,x) \in L^\infty_c(\mathcal{I} \times \mathbb{R}^d)\), we therefore obtain
		
		\begin{align*}
			\Bigg|\operatorname{Tr}\left( \int_{\mathcal{I}} e^{-it\mathcal{H}^+} V(t,x) e^{it\mathcal{H}^+} \, dt \,  \right)\gamma\Bigg|
			&= \Bigg|\int_{\mathcal{I}} \operatorname{Tr}\big( V e^{it\mathcal{H}^+} \gamma e^{-it\mathcal{H}^+} \big) \, dt \Bigg|\\
			&= \Bigg|\int_{\mathcal{I}} \int_{\mathbb{R}^d} V(t,x) \rho_{\gamma(t)}(x) \, dx \, dt\Bigg| \\
			&\le \| V \|_{L^{q'}_t(\mathcal{I}, \mathcal{W}^{p,p'}_x(\mathbb{R}^d))} \| \rho_{\gamma(t)} \|_{L^{q}_t(\mathcal{I}, \mathcal{W}^{p',p}_x(\mathbb{R}^d))}\\
			&\leq \| V \|_{L^{q'}_t(\mathcal{I}, \mathcal{W}^{p,p'}_x(\mathbb{R}^d))}\|\gamma\|_{\mathfrak{S}^{\frac{2p}{p+1}}}.
		\end{align*}
		Using H\"older's inequality and inequality \eqref{equivalentONS}. Therefore,
		\begin{align}\label{dualversionofH+ONS}
			\Bigg\|\int_{\mathcal{I}} e^{-it\mathcal{H}^+} V(t,x) e^{it\mathcal{H}^+} \, dt\Bigg\|_{\mathfrak{G}^{2p'}}\lesssim\| V \|_{L^{q'}_t(\mathcal{I}, \mathcal{W}^{p,p'}_x(\mathbb{R}^d))}.
		\end{align}
		The above inequality is called the dual inequality of \eqref{ONShermite1}. This will be instrumental in proving the end point case $p=\frac{d+1}{d-1}$ of Theorem \ref{ONSH+improved}. The next Proposition will show that the orthonormal inequality \eqref{ONShermite1} does not hold for $p=\frac{d+1}{d-1}$. This is equivalent to show that we can always choose a potential $V\in L_c^{\infty}(\mathcal{I}\times \mathbb{R}^d)$ such that  $\Bigg\|\int_{\mathcal{I}} e^{-it\mathcal{H}^+} V(t,x) e^{it\mathcal{H}^+} \, dt\Bigg\|_{\mathfrak{G}^{d+1}}=\infty$.
		\begin{proposition}[The endpoint]\label{endpointproposition}
			
			\[
			\operatorname{Tr} \left( \int_{\mathcal{I}} e^{-it\mathcal{H}^{+}} V(t,x) e^{it\mathcal{H}^{+}} \, dt \right)^{d+1} = +\infty,
			\]
			for some $V\in L_c^{\infty}(\mathcal{I}\times \mathbb{R}^d)$.
		\end{proposition}
		\begin{proof}
			Let us define the operator $B_{V}$ for any $V\in L_c^{\infty}(\mathcal{I}\times \mathbb{R}^d)$ by
			\begin{align}\label{Kernel BV}
				B_{V}:=\int_{\mathcal{I}}e^{-it\mathcal{H}^{+}} V(t,x) e^{it\mathcal{H}^{+}} \, dt.
			\end{align}
			Since $\mathcal{H}^{+}=\mathcal{H}_{\mathcal{A}}$ for $\mathcal{A}=\begin{pmatrix} 0 & I_d \\ -I_d & 0 \end{pmatrix}$, we have 
			\begin{align*}
				e^{t\mathcal{A}}=\begin{pmatrix} \cos t & \sin t \\ -\sin t & \cos t \end{pmatrix}.
			\end{align*}
			From Proposition \ref{propmetaplectic}, we have 
			\begin{align*}
				e^{it\mathcal{H}^{+}}f(x)=\mu(e^{t\mathcal{A}})=\left(\frac{1}{2 \pi i \sin 2t}\right)^{d/2}\int_{\mathbb{R}^d}e^{\frac{i}{\sin 2t}(|x|^2\cos 2t-2x.y+|y|^2\cos 2t)}f(y)\, dy.
			\end{align*}
			Using the above expression,  the kernel $\mathcal{K}_{B_V}$ of the operator $B_V$ can be written as 
			\begin{align*}
				\mathcal{K}_{B_V}(x,z)&=\int_{\mathbb{R}^{d}}\int_{\mathcal{I}}K_{it}(x,z,t)V(t,x)\overline{K_{it}(y,z,t)} dy\, dt,
			\end{align*}
			where $K_{it}$ denoted the kernel of the Harmonic oscillator $\mathcal{H}^{+}$. The Fourier transform of $\mathcal{K}_{B_V}$  with respect to the both variables  yields
			\begin{align*}
				\widehat{\mathcal{K}_{B_V}}(p, q)&=\int_{\mathbb{R}^d}\int_{\mathbb{R}^d}e^{-ip.x}\mathcal{K}_{B_V}(x,y)e^{iq.y}dy\, dx\\
				&=\int_{\mathcal{I}}dt\, \int_{\mathbb{R}^d}dz \, V(t,z)\Bigg(\int_{\mathbb{R}^d}e^{-ip.x}K_{it}(x,z,t)\, dx\Bigg)\Bigg(\int_{\mathbb{R}^d}e^{iq.y}K_{-it}(y,z,t)\, dy\Bigg).
			\end{align*}
			Denote
			\begin{align*}
				I_x(p,z):=\int_{\mathbb{R}^d}e^{-ip.x}K_{it}(x,z,t)\, dx \quad I_y(q,z):=\int_{\mathbb{R}^d}e^{iq.y}K_{-it}(y,z,t)\, dy.
			\end{align*}
			Now we want to evaluate $ I_x(p,z)$ and $I_y(q,z)$.
			\begin{align*}
				I_x(p,z)&=\int_{\mathbb{R}^d}\left(\frac{1}{2\pi i \sin 2t}\right)^{d/2}e^{-ip.x}e^{\frac{i}{\sin 2t}(|x|^2\cos 2t-2x.z+|z|^2\cos 2t)}\, dx\\
				&=\frac{e^{\frac{i}{2}|z|^2 \cot2t}}{(2 \pi i \sin 2t)^{d/2}}\int_{\mathbb{R}^d}e^{-ip.x}e^{\frac{i}{2}|x|^2\cot 2t-\frac{ix.z}{\sin2t}}\, dx\\
				&=\frac{1}{(\cos 2t)^{d/2}}e^{\frac{i}{2}|z|^2\cot 2t-\frac{i}{2}(\tan{2t})|p+\frac{z}{\sin 2t}|^2}.
			\end{align*}
			Similarly we can evaluate 
			\begin{align*}
				I_y(q,z)=\frac{1}{(\cos 2t)^{d/2}}e^{\frac{-i}{2}|z|^2\cot 2t+\frac{i}{2}(\tan{2t})|q+\frac{z}{\sin 2t}|^2}.
			\end{align*}
			Putting all together, we get
			\begin{align*}
				\widehat{\mathcal{K}_{B_V}}(p, q)&=\int_{\mathcal{I}}\int_{\mathbb{R}^d}\frac{e^{-i(\tan 2t)(|p|^2-|q|^2)}}{|\cos 2t|^{d}}V(t,z)e^{-i(\sec 2t) z.(p-q)}\, dz \, dt\\
				&=:\int_{\mathcal{I}}\frac{e^{-i(\tan 2t)(|p|^2-|q|^2)}}{|\cos 2t|^{d}} \Bigg(\int_{\mathbb{R}^d}V_1(\tan2t,z)e^{-i(\sec 2t)(p-q).z}\, dz\Bigg) \, dt,
			\end{align*}

			Now we can do change of variable $(\sec 2t)z=u$ and subsequently $\tan2t=s$ to get 
			\begin{align*}
				\widehat{\mathcal{K}_{B_V}}(p, q)&=\int_{\mathbb{R}}\int_{\mathbb{R}^d}\frac{e^{-is(|p|^2-|q|^2)}}{2(1+s^2)} V_1\left(s,\frac{u}{\sqrt{1+s^2}}\right)e^{-i(p-q).u}\, du \, ds \\
				&=:\int_{\mathbb{R}}\int_{\mathbb{R}^d}e^{-is(|p|^2-|q|^2)} V_2(s,u) e^{-i(p-q).u}\, du \, ds \\
				&=\hat{V_2}(|p|^2-|q|^2,p-q).
			\end{align*}
			We choose $V_2\in L_c^{\infty}(\mathbb{R}\times \mathbb{R}^d)$ such that $V_2>0$ and $\hat{V_2}>0$. Clearly for such choice of $V_2$, we also have $V\in L_c^{\infty}(\mathcal{I}\times \mathbb{R}^d)$. Following an approach similar to that used in the proof of Proposition 2 in \cite{frank}, we conclude that
			\[
			\operatorname{Tr} \left( \int_{\mathcal{I}} e^{-it\mathcal{H}^{+}} V(t,x) e^{it\mathcal{H}^{+}} \, dt \right)^{d+1} = +\infty.
			\]
		\end{proof}

		\section{Orthonormal Strichartz estimates for \texorpdfstring{$\mathcal{H}_{\mathcal{A}}$}{}}\label{sectionHA}
		Let \(\mathcal{A} = \begin{pmatrix} 0 & B \\ 0 & 0 \end{pmatrix} \in \mathrm{Sp}(d,\mathbb{R})\), where \(B = B^*\) is a real symmetric matrix. The corresponding Hamiltonian operator is given by
		\[
		\mathcal{H}_{\mathcal{A}} = -\frac{1}{4\pi} B \nabla \cdot \nabla,
		\]
		and its exponential yields
		\[
		e^{t\mathcal{A}} = \begin{pmatrix} I_d & tB \\ 0 & I_d \end{pmatrix} \in \mathrm{Sp}(d,\mathbb{R}).
		\]
		
		Fix \(t \neq 0\). Assuming \(\det B \neq 0\) and denoting the eigenvalues of \(B\) by \(\lambda_1, \ldots, \lambda_d\), the fixed-time estimate reads as follows.
		
		\begin{align}
			\| e^{it\mathcal{H}_{\mathcal{A}}} f \|_{\mathcal{W}^{1,\infty}}
			\le \prod_{j=1}^d \left( \frac{1 + t^2 \lambda_j^2}{t^4 \lambda_j^4} \right)^{1/4} \| f \|_{\mathcal{W}^{\infty,1}}.\label{disestgenlap}
		\end{align}
		
		Using the above estimate, we now turn to the proof of the single-function Strichartz estimate for the operator \(\mathcal{H}_{\mathcal{A}}\) in the Wiener amalgam setting.
		
		\begin{theorem}[Strichartz estimate for single function]
			\label{thm:strichartzz}
			Let $4 < q \leq \infty$, $2 \leq p\leq \infty$, such that
			\begin{equation}\nonumber
				\frac{2}{q} + \frac{d}{p} = \frac{d}{2}.
			\end{equation}
			Then we have the Strichartz estimates
			\begin{equation}
				\|e^{it\mathcal{H}_{\mathcal{A}}} f\|_{\mathcal{W}_t{(L^{q/2},L^q)} \mathcal{W}^{r'r}_x} \lesssim \|f\|_{L^2_x}.\label{strichartzforHA}
			\end{equation}
			
		\end{theorem}
		\begin{proof} Let us consider a function $\phi$ by 
			
			$$ \phi(t)=\prod_{j=1}^d \left( \frac{1 + t^2 \lambda_j^2}{t^4 \lambda_j^4} \right)^{1/4} \lesssim\begin{cases}
				\frac{1}{t^d}, & \text{for}~~ 0<|t|<1,\\
				\frac{1}{t^{d/2}} &  \text{for}~~ t\geq 1.
			\end{cases}$$
			Then 
			\begin{align*}
				\phi(t)\lesssim \frac{1}{t^d}.\chi_{\{0<|t|<1\}}(t)+\frac{1}{t^{d/2}}. \chi_{\{1\leq |t|\}}(t)\leq \frac{1}{t^d}+\frac{1}{t^{d/2}}.
			\end{align*}
			Applying interpolation between \eqref{disestgenlap} and the energy estimate $\|e^{it\mathcal{H}_{\mathcal{A}}} f\|_{W^{2,2}}=\|f\|_{W^{2,2}}$, we have
			\begin{align}
				\|e^{it\mathcal{H}_{\mathcal{A}}} f\|_{\mathcal{W}^{p',p}}\lesssim \phi(t)^{(1-\frac{2}{p})}\|f\|_{\mathcal{W}^{p,p'}}\lesssim \psi_{\frac{d}{2}(1-\frac{2}{p})}(t)\|f\|_{\mathcal{W}^{p,p'}},\label{disestgenrallaplacian}
			\end{align}
			where $\psi_{\sigma}$ is defined by \eqref{Phidefn.}. Set $\alpha = d\left(\frac{1}{2} - \frac{1}{p}\right) = \frac{2}{q}.$ Now we use Inequality \eqref{disestgenrallaplacian} and Lemma \ref{phispispaces} to obtain the following: 
			\begin{align}\nonumber
				\left\| \int e^{i(t-s)\mathcal{H}_{\mathcal{A}}} F(s) \, ds \right\|_{\mathcal{W}_t{(L^{q/2},L^q)} \mathcal{W}^{p',p}_x} \lesssim \|F\|_{\mathcal{W}(L^{({q}/2)'},L^{{q}'}) \mathcal{W}^{p,p'}_x},
			\end{align}
			which is enough to show the required result by $TT^*$ method. This completes the proof of the Theorem.
		\end{proof}
		The proof of the following strichartz estimate for the endpoint $P := (4, 2d/(d - 1))$ follows from the arguments in \cite[Theorem 5.2]{NFC} and \cite{KT}.

		\begin{theorem}
			Let $d>1$, then
			\begin{align}\label{endpointsingstrichartz}
				\|e^{it\mathcal{H}_{\mathcal{A}}}f\|_{\mathcal{W}_t(L^2,L^4)(\mathcal{W}^{\frac{2d}{d+1},\frac{2d}{d-1}})}\lesssim \|f\|_{L^2(\mathbb{R}^d)}.
			\end{align}
		\end{theorem}
		
		\subsection{Orthonormal  estimate}
		For $A=e^{it\mathcal{H}_{\mathcal{A}}}$,  the operator $T$ 
		\begin{align*}
			T=AA^*:L^{(\frac{q}{2})'}\bigl(\mathbb{R}, \mathcal{W}^{p,p'}(\mathbb{R}^d)\bigr)\to L^{\frac{q}{2}}\bigl(\mathbb{R}, \mathcal{W}^{p',p}(\mathbb{R}^d)\bigr)
		\end{align*}
		defined as 
		\begin{align}\nonumber
			TF(t,x)&:=\int_{\mathbb{R}}e^{i(t-s)\mathcal{H}_{\mathcal{A}}}F(s,x)ds\\
			&=\int_{\mathbb{R}}\int_{\mathbb{R}^d}K_{i(t-s)}(x,y)F(s,x)\,dy\,ds,\label{DualequalityforeitHA}
		\end{align}
		where,
		\begin{align*}
			K_{it}(x,y)=(-i \det (tB))^{-1/2}e^{i \pi (x-y)(tB)^{-1}(x-y)}.
		\end{align*}
		Then  the above kernel $K_{it}$ satisfies the following  uniform estimate:
		\begin{align}\nonumber
			\left| K_{it}(x, y) \right| \lesssim \frac{1}{|t|^{d/2}}.
		\end{align}
		Moreover, we can further reduce this estimate to one of its weak forms, 
		\begin{align}\label{kernelweakdecay}
			\left| K_{it}(x, y) \right| \lesssim \frac{1}{|t|^{d/2}}\chi_{\{|t|> 1\}}(t)+\frac{1}{|t|^{d}}\chi_{\{|t|\leq 1\}}(t)\lesssim \psi_{d/2}(t),
		\end{align}
		and using the above weak form of the kernel, we have the following result in Schatten spaces, which is crucial to prove orthonormal Strichartz estimates for $\mathcal{H}_A.$

		\begin{theorem}\label{theoremWTWHA}
			Let   $ 2\leq p,q,\tilde{q} \leq \infty $ such that 
			\[\frac{1}{\tilde{q}}=\frac{2}{q}-\frac{1}{2}, \quad  \frac{2}{q} + \frac{1}{p} = 1 \quad \text{and} \quad \frac{5}{2} < p < 3,
			\]
			then  
			\begin{align}
				\| W_1 T W_2 \|_{\mathfrak{S}^p(L^2(\mathbb{R}, L^2(\mathbb{R})))} \leq C \| W_1 \|_{\mathcal{W}_t(L^{\tilde{q}},L^{q})(\mathbb{R},  \mathcal{W}_x^{p',p}(\mathbb{R}^d))} \| W_2 \|_{\mathcal{W}_t(L^{\tilde{q}},L^{q})(\mathbb{R},  \mathcal{W}_x^{p',p}(\mathbb{R}^d))},
			\end{align}
			where $W_1$ and $W_2$ should be interpreted as multiplication operator corresponding to the function $W_1$ and $W_2$, respectively.
		\end{theorem}

		\begin{proof} 
			For $\epsilon>0$, we define 
			\[T_{\epsilon}F(t,x)=\int_{\mathbb{R}}\int_{\mathbb{R}}K_{\epsilon}(x,y,t-s)F(s,y)\,dy\,ds,\]
			where,
			\[K_{\epsilon}(x,y,t)=\chi_{\epsilon<|t|< \infty}(t)K_{it}(x,y).\]
			To use Stein's analytic complex interpolation, for $z \in \mathbb{C}$, we define
			\[
			K_{z,\varepsilon}(x, y, t) = t^{z} K_{\varepsilon}(x, y, t)
			\]
			and define the corresponding operator as
			\[
			T_{z,\varepsilon} F(t,x) = \int_{\mathbb{R}}\int_{\mathbb{R}}  K_{z,\varepsilon}(x, y, t - s) F(s,y)  \, dy \, ds.
			\]
			Then for any $\varepsilon > 0$, $\{T_{z,\varepsilon}\}_z$ forms an analytic family of operators in the sense of Stein. Moreover, we can deduce that for $\operatorname{Re}(z)\leq \frac{1}{2}$, we have
			
			\begin{align*}
				|K_{z,\varepsilon}(x, y, t)| & \lesssim \frac{1}{|t|^{\frac{1}{2}-\operatorname{Re}(z)}}=\frac{1}{|t|^{\frac{1}{2}-\operatorname{Re}(z)}}\chi_{\{|t|\leq 1\}}+\frac{1}{|t|^{\frac{1}{2}-\operatorname{Re}(z)}}\chi_{\{|t|> 1\}}\\
				& \leq \frac{1}{|t|^{1-2\operatorname{Re}(z)}}\chi_{\{|t|\leq 1\}}+\frac{1}{|t|^{\frac{1}{2}-\operatorname{Re}(z)}}\chi_{\{|t|> 1\}}\lesssim \psi_{\frac{1}{2}-\operatorname{Re}(z)}(t),
			\end{align*}
			where $\psi_{{\frac{1}{2}-\operatorname{Re}(z)}}$ is defined by \eqref{Phidefn.}. Let $\operatorname{Re}(z)=\lambda$ and denote $\frac{1}{2}-\lambda$ by $\sigma$. Then, 
			\begin{align}\nonumber
				\|W_1 T_{\lambda+ia,\varepsilon} W_2 \|^2_{\mathfrak{S}^2(L^2(\mathbb{R}^{1+1}))} &= \int_{\mathbb{R} \times \mathbb{R}} \int_{\mathbb{R} \times \mathbb{R}} |W_1(t,x)|^2 |K_{\lambda+ia,\varepsilon}(x, y, t-s)|^2 |W_2(s,y)|^2 \, dx \, dy \, dt \, ds\\\nonumber
				& \lesssim \int_{\mathbb{R} \times \mathbb{R}} \int_{\mathbb{R}\times \mathbb{R}} {|W_1(t,x)|^2 |W_2(s,y)|^2} \psi_{2\sigma}(t-s) \, dx \,  dy \, dt \, ds\\\nonumber
				& \lesssim \int_{\mathbb{R}}\|W_1(t,.)\|^2_{\mathcal{W}_x^{2,2}}(\|W_2\|^2_{\mathcal{W}_x^{2,2}}*\psi_{2\sigma})(t)\,dt \\ 
				& \lesssim\|W_1\|^2_{\mathcal{W}_t(L^{\frac{2u}{2-u}},L^{2u})(\mathbb{R},\mathcal{W}_x^{2,2}(\mathbb{R}))} \|W_2\|^2_{\mathcal{W}_t(L^{\frac{2u}{2-u}},L^{2u})(\mathbb{R},\mathcal{W}_x^{2,2}(\mathbb{R}))},\label{schatten2inequalityHA23}
			\end{align}
			provided that $0< 2\sigma <\frac{1}{2}$, $\frac{2}{{u}}+2\sigma=2$. The last inequality is obtained by applying H\"{o}lder's inequality followed by the application of the Lemma \ref{youngtypeinequality} for $d=1$. Since \(0 < 2\sigma < \frac{1}{2}\), the relation \(\frac{2}{u} + 2\sigma = 2\) implies \(1 < u < \frac{4}{3}\). Consequently, we have the inequality
			\[
			\frac{2u}{2-u} \le \frac{2u}{2 - (\lambda + 1)u}=\frac{2}{\lambda}.
			\]
			Now, we use the inclusion property from the Proposition \ref{waspaceproperties} to have the following inequality 
			\begin{equation}\label{Interpolation12}
				\|W_1 T_{\lambda+ia,\varepsilon} W_2 \|_{\mathfrak{S}^2(L^2(\mathbb{R}^{1+1}))}\lesssim\|W_1\|_{\mathcal{W}_t(L^{\frac{2}{\lambda}},L^{\frac{4}{2\lambda+1}})(\mathbb{R},\mathcal{W}_x^{2,2}(\mathbb{R}))} \|W_2\|_{\mathcal{W}_t(L^{\frac{2}{\lambda}},L^{\frac{4}{2\lambda+1}})(\mathbb{R},\mathcal{W}_x^{2,2}(\mathbb{R}))} .
			\end{equation}
			We aim to show that for $\operatorname{Re}(z) = -1$, the operator $T_{z,\varepsilon} : L^{2}\bigl(\mathbb{R}, \mathcal{W}^{2,2}(\mathbb{R})\bigr) \to L^{2}\bigl(\mathbb{R}, \mathcal{W}^{2,2}(\mathbb{R})\bigr)$ is bounded with some constant depending only on $d$ and $|\operatorname{Im}(z)|$ exponentially. Notice that 
			\begin{align*}
				T_{z,\varepsilon} F(t,x) &= \int_{\mathbb{R}} \int_{\mathbb{R}} K_{z,\varepsilon}(x, y, t-s) F(s,y) \, ds \, dy\\
				&= \int_{\varepsilon < |s| <\infty} s^{-1 + i \operatorname{Im}(z)}  \left(\int_{\mathbb{R}} K_{i s}(x, y) F(t-s,y) \, dy\right) \, ds\\
				&= \int_{\varepsilon < |s| <\infty}s^{-1 + i \operatorname{Im}(z)} e^{i s \mathcal{H}_{\mathcal{A}}} F(t-s,x) \, ds.
			\end{align*}
			
			For fixed $z$ and $t$, we can write 
			\begin{align*}
				\left(e^{-i t \mathcal{H}_{\mathcal{A}}}T_{z,\varepsilon} F\right)(t,x)& =\int_{\mathbb{R}}K_{-it}(x,y)T_{z,\varepsilon} F(t,y)\, dy\\
				&= \int_{\mathbb{R}}K_{-it}(x,y)\left(\int_{\mathbb{R}} \int_{\mathbb{R}} K_{z,\varepsilon}(y, w, s) F(t-s,w) \, ds \, dw\right)\, dy \\
				&=\int_{\varepsilon < |s| <\infty} s^{-1 + i \operatorname{Im}(z)} \left(\int_{\mathbb{R}}K_{-it}(x,y) e^{i s \mathcal{H}_{\mathcal{A}}} F(t-s,y)\, dy\right) \, ds \\
				&=\int_{\varepsilon < |s| <\infty}s^{-1 + i \operatorname{Im}(z)} e^{-i (t-s) \mathcal{H}_{\mathcal{A}}} F(t-s,x) \, ds\\
				&=: \int_{\varepsilon < |s| <\infty} s^{-1 + i \operatorname{Im}(z)} G(t-s,x) \, ds,
			\end{align*}
			where $G(t,x) = e^{-i t \mathcal{H}_{\mathcal{A}}} F(t,x)$. Therefore
			\begin{align}\label{normequalityeitHTF12}
				\|e^{-i t \mathcal{H}_{\mathcal{A}}}T_{z,\varepsilon} F\|_{\mathcal{W}^{2,2}(\mathbb{R})}
				= \left\| \int_{\varepsilon < |s| <\infty} s^{-1 + i \operatorname{Im}(z)} G(t-s,\cdot) \, ds \right\|_{\mathcal{W}^{2,2}(\mathbb{R})},
			\end{align}
			
			and since for fixed $t$ the operator $e^{-i t \mathcal{H}_{\mathcal{A}}}:L^2(\mathbb{R})\to  L^2(\mathbb{R})$ is unitary,
			it follows that 
			\begin{align}
				\| G(t,\cdot) \|_{\mathcal{W}^{2,2}(\mathbb{R})} = \| F(t, \cdot) \|_{L^2(\mathbb{R})}.\label{normequalityGF} 
			\end{align}
			If we define
			\[
			\mathcal{H}_{z,\varepsilon} : h(t) \mapsto \int_{\varepsilon < |s| <\infty} s^{-1 + i \operatorname{Im}(z)} h(t-s) \, ds,
			\]
			then using \eqref{normequalityeitHTF12}, we get
			
			\[\|e^{-i t \mathcal{H}_{\mathcal{A}}}T_{z,\varepsilon} F\|_{\mathcal{W}^{2,2}(\mathbb{R})}=\|\mathcal{H}_{z,\varepsilon}G\|_{\mathcal{W}^{2,2}(\mathbb{R})}.\]
			This eventually gives  
			\begin{align}
				\| {T}_{z,\varepsilon} F \|_{L^2(\mathbb{R}, \mathcal{W}^{2,2}(\mathbb{R}))} = \|\mathcal{H}_{z,\varepsilon}G\|_{L^2(\mathbb{R},{\mathcal{W}^{2,2}(\mathbb{R}))}}=\| \mathcal{H}_{z,\varepsilon} G \|_{L^2(\mathbb{R}, L^2(\mathbb{R}))}.\label{normequalityTH1}
			\end{align}
			Since the operator $\mathcal{H}_{z,\varepsilon}$ is just Hilbert transform up to $i \operatorname{Im}(z)$, we have the bound $\mathcal{H}_{z,\varepsilon} : L^2(\mathbb{R}) \to L^2(\mathbb{R})$ with some constant depending only on $\operatorname{Im}(z)$ exponentially. Now combining \eqref{normequalityGF} and \eqref{normequalityTH1}, we obtain
			\[
			\| {T}_{z,\varepsilon} F \|_{L^2(\mathbb{R}, \mathcal{W}^{2,2}(\mathbb{R}))} \leq C(\operatorname{Im}(z)) \| F \|_{L^2(\mathbb{R}, \mathcal{W}^{2,2}(\mathbb{R}))},
			\]
			which is our desired bound for ${T}_{z,\varepsilon} : L^2(\mathbb{R}, \mathcal{W}^{2,2}(\mathbb{R})) \to L^2(\mathbb{R}, \mathcal{W}^{2,2}(\mathbb{R}))$ when $\operatorname{Re}(z) = -1$.
			Noting the fact that $\mathfrak{S}^\infty$-norm is the usual operator norm, for $\operatorname{Re}(z) = -1$, we get
			\begin{align}\nonumber
				\| W_1 {T}_{z,\varepsilon} W_2 \|_{\mathfrak{S}^\infty(L^2(\mathbb{R}\times \mathbb{R}))}&=\| W_1 {T}_{z,\varepsilon} W_2 \|_{L^2(\mathbb{R}\times \mathbb{R})\to L^2(\mathbb{R}\times \mathbb{R})} \\ 
				&\leq C(\operatorname{Im}(z))\| W_1\|_{\operatorname{op}} \| W_2\|_{\operatorname{op}} \\
				%&\leq C(\operatorname{Im}(z)) \| W_1\| \| W_2\|\\ 
				&\leq C(\operatorname{Im}(z)) \| W_1\|_{\mathcal{W}_t({L^{\infty},L^{\infty}})(\mathbb{R},\mathcal{W}^{1, \infty}(\mathbb{R}))} \| W_2\|_{\mathcal{W}_t({L^{\infty},L^{\infty}})(\mathbb{R},\mathcal{W}^{1, \infty}(\mathbb{R}))}.\label{infinitynormofWTW}
			\end{align}
			We use complex interpolation between \eqref{Interpolation12} and \eqref{infinitynormofWTW} for $z=0$, we obtain the following inequality
			\begin{align}
				\| W_1 {T}_{\varepsilon} W_2 \|_{\mathfrak{S}^p(L^2(\mathbb{R}\times \mathbb{R}))}\lesssim \|W_1\|_{\mathcal{W}_t(L^{\tilde{q}},L^q)\mathcal{W}^{p',p}}\|W_2\|_{\mathcal{W}_t(L^{\tilde{q}},L^q)\mathcal{W}^{p',p}},\label{HAWTWinequlityfinal}
			\end{align}
			where $\tilde{q},q$ and $p$ satisfy the relations  $ 2\leq p,q,\tilde{q} \leq \infty $ such that 
			\[\frac{1}{\tilde{q}}=\frac{2}{q}-\frac{1}{2}, \quad  \frac{2}{q} + \frac{1}{p} = 1 \quad \text{and} \quad \frac{5}{2} < p < 3.
			\] 
			The above inequality is independent of $\varepsilon$. Finally, letting $\varepsilon\rightarrow 0^+$ in (\ref{HAWTWinequlityfinal}), we complete the proof of the theorem.
		\end{proof}
		\begin{proof}[Proof of the theorem \ref{theoremONSHA}]\label{proofONSHA}
			
			The proof proceeds analogously to that of Theorem \ref{ONSH+improved}. Invoking Theorem \ref{theoremWTWHA} together with the duality principle \ref{DualityWAspace}, and using the trivial estimate
			\begin{equation*}
				\left\|\sum_{j\in J}\alpha_j|e^{it\mathcal{H}_{\mathcal{A}}}f_j|^2\right\|_{\mathcal{W}_t(L^\infty,L^\infty)\mathcal{W}_x^{\infty,1}} \leq \sum_{j\in J}|\alpha_j|,
			\end{equation*}
			valid for any (possibly infinite) orthonormal system \(\{f_j\}_{j\in J}\) in \(L^2(\mathbb{R})\) and any sequence \(\{\alpha_j\}_{j\in J}\) in \(\mathbb{C}\), we obtain the desired result.
		\end{proof}

		\begin{remark}
			
			Unlike the Strichartz estimate \eqref{CorNicostrichartz}, which is measured in $L^{q/2}$ in time, Theorem \ref{strichartzH+improved} provides a refined single-function estimate with time exponent $q$. By pursuing a similar strategy, we deduce the following version of the Strichartz estimate for the operator $\mathcal{H}_{\mathcal{A}}$
			\begin{equation}
				\|e^{it\mathcal{H}_{\mathcal{A}}} f\|_{L^q \mathcal{W}^{p',p}_x} \lesssim \|f\|_{L^2_x},\label{strichartzforHA2}
			\end{equation}
			where the time variable belongs to a Lebesgue space instead of a Wiener amalgam space, following the same admissible condition in $q$ and $p$ as in \eqref{thm:strichartzz} but with the restriction $2\leq p <\frac{2(d+1)}{d-1}.$
			Since $L^q=\mathcal{W}(L^q,L^q)\hookrightarrow \mathcal{W}(L^{q/2},L^q)$ for every $2\leq q\leq \infty$, the above estimate \eqref{strichartzforHA2}  is an improvement  of Theorem \eqref{thm:strichartzz} for $\mathcal{H}_{\mathcal{A}}$. Consequently, the same argument for orthonormal estimates yields the corresponding orthonormal Strichartz estimate for \eqref{strichartzforHA2} in every dimension $d\geq1.$ Furthermore, by arguments analogous to those used in Propositions \ref{optimalityschattenexponent} and \ref{endpointproposition}, one can establish the corresponding optimality and endpoint results for the operator $\mathcal{H}_{\mathcal{A}}$.
		\end{remark}

		\section{Orthonormal Strichartz estimates for  \texorpdfstring{$\mathcal{H}^-$}{}}\label{sectionH-}
		Consider the case of the inverted harmonic oscillator  $ 
		\mathcal{H}^{-} = -\frac{1}{4\pi}\Delta - \pi |x|^2.$ 
		This corresponds to the matrix
		\[
		\mathcal{A} = \begin{pmatrix} 0 & I_d \\ I_d & 0 \end{pmatrix} \in \mathfrak{sp}(d, \mathbb{R}),
		\]
		and the corresponding symplectic flow is given by
		\[
		e^{t\mathcal{A}} =
		\begin{pmatrix}
			\cosh t  & \sinh t  \\
			\sinh t  & \cosh t 
		\end{pmatrix} \in \operatorname{Sp}(d, \mathbb{R}).
		\]
		Then we can use equality \eqref{represetationofmuA} to get the following 
		
		\begin{align}\label{IntergralrepH-}
			e^{it\mathcal{H}^{-}} f(x) = i^{d/2} (\sinh t)^{-d/2}
			\int_{\mathbb{R}^d}
			e^{-\pi i (\coth t)(|x|^2 + |y|^2) + 2\pi i (\operatorname{cosech} t) y \cdot x}
			f(y) \, dy.  
		\end{align}
		Suppose $A=e^{it\mathcal{H}^{-}}$, we can define the operator 
		\begin{align*}
			T=AA^*:\mathcal{W}_t(L^{(q/2)'}, L^2) \mathcal{W}_x^{p,p'}\to \mathcal{W}_t(L^{q/2}, L^2) \mathcal{W}_x^{p',p}
		\end{align*}
		as
		\begin{align}\nonumber
			TF(t,x)&:=\int_{\mathbb{R}}e^{i(t-s)\mathcal{H}^{-}}F(s,x)ds\\
			&=\int_{\mathbb{R}}\int_{\mathbb{R}^d}K_{i(t-s)}(x,y)F(s,x)\,dy\,ds,\label{TFinvertedhermitedr}
		\end{align}
		where $K_{it}$ is the kernel of the operator $e^{it\mathcal{H}^{-}}$, whose expression can be obtained from the representation \eqref{IntergralrepH-}. Consequently, we have the following uniform decay estimates of the kernel
		\begin{align}
			|K_{it}(x,y)|\lesssim \frac{1}{|\sinh{t}|^{d/2}}.
		\end{align}

		\begin{theorem}\label{WTWtheoremstatH-}
			For  $ 2\leq p,q, \leq \infty $ such that 
			\[ \frac{1}{q} + \frac{1}{p} = \frac{1}{2} \quad \text{and} \quad \frac{5}{2} < p < 3,
			\]
			we have
			
			\begin{align}\label{WTWinqualityforHA}
				\| W_1 T W_2 \|_{\mathfrak{S}^p(L^2(\mathbb{R}, L^2(\mathbb{R})))} \leq C \| W_1 \|_{\mathcal{W}_t(L^{{q}},L^{4})(\mathbb{R},  \mathcal{W}_x^{p',p}(\mathbb{R}))} \| W_2 \|_{\mathcal{W}_t(L^{{q}},L^{4})(\mathbb{R},  \mathcal{W}_x^{p',p}(\mathbb{R}))},
			\end{align}
			where $W_1$ and $W_2$ should be interpreted as multiplication operator corresponding to the function $W_1$ and $W_2$, respectively.
		\end{theorem}
		
		\begin{proof}
			For $\epsilon>0$, we define 
			\[T_{\epsilon}F(t,x)=\int_{\mathbb{R}}\int_{\mathbb{R}}K_{\epsilon}(x,y,t-s)F(s,y)\,dy\,ds,\]
			where,
			\[K_{\epsilon}(x,y,t)=\chi_{\epsilon<|t|< \infty}(t)K_{it}(x,y).\]
			
			To use Stein's analytic complex interpolation, for $z \in \mathbb{C}$, we define
			\[
			K_{z,\varepsilon}(x, y, t) = (\sinh{t})^z K_{\varepsilon}(x, y, t)
			\]
			and define the corresponding operator as
			\[
			T_{z,\varepsilon} F(t,x) = \int_{\mathbb{R}}\int_{\mathbb{R}}  K_{z,\varepsilon}(x, y, t - s) F(s,y)  \, dy \, ds.
			\]
			Then for any $\varepsilon > 0$, $\{T_{z,\varepsilon}\}_z$ forms an analytic family of operators in the sense of Stein. Moreover, we can deduce that for $\operatorname{Re}(z)\leq \frac{1}{2}$, we have
			
			\begin{align*}
				|K_{z,\varepsilon}(x, y, t)| & \lesssim \frac{1}{|\sinh{t}|^{\frac{1}{2}-\operatorname{Re}(z)}}=\frac{1}{|\sinh{t}|^{\frac{1}{2}-\operatorname{Re}(z)}}\chi_{\{|t|\leq \sinh^{-1}{1}\}}+\frac{1}{|\sinh{t}|^{\frac{1}{2}-\operatorname{Re}(z)}}\chi_{\{|t|> \sinh^{-1}{1}\}}\\
				& \leq \frac{1}{|\sinh{t}|^{1-2\operatorname{Re}(z)}}\chi_{\{|{t}|\leq \sinh^{-1}{1}\}}+\frac{1}{|\sinh{t}|^{\frac{1}{2}-\operatorname{Re}(z)}}\chi_{\{|t|> \sinh^{-1}{1}\}}\lesssim \varphi_{\frac{1}{2}-\operatorname{Re}(z)}(t),
			\end{align*}
			where $\varphi_{\alpha}$ is defined in \eqref{vaphidef.}. Let $\operatorname{Re}(z)=\lambda$ and denote $\frac{1}{2}-\lambda$ by $\sigma$. Successive application of the H\"older's inequality and the Lemma \ref{youngtypeinequality} gives
			\begin{align}\nonumber
				\|W_1 T_{\lambda+ia,\varepsilon} W_2 \|^2_{\mathfrak{S}^2(L^2(\mathbb{R}^{1+1}))} &= \int_{\mathbb{R} \times \mathbb{R}} \int_{\mathbb{R} \times \mathbb{R}} |W_1(t,x)|^2 |K_{\lambda+ia,\varepsilon}(x, y, t-s)|^2 |W_2(s,y)|^2 \, dx \, dy \, dt \, ds\\\nonumber
				& \lesssim \int_{\mathbb{R} \times \mathbb{R}} \int_{\mathbb{R} \times \mathbb{R}} {|W_1(t,x)|^2 |W_2(s,y)|^2}\varphi_{2\sigma}(t-s) \, dx \,  dy \, dt \, ds\\\nonumber
				& \lesssim \int_{\mathbb{R}}\|W_1(t,.)\|^2_{\mathcal{W}_x^{2,2}}(\|W_2\|^2_{\mathcal{W}_x^{2,2}}*\varphi_{2\sigma})(t)\,dt \\ \nonumber
				& \lesssim \bigg\|\|W_1(t,.)\|^2_{\mathcal{W}_x^{2,2}}\Bigg\|_{\mathcal{W}_t(L^{u},L^{2})}   \bigg\|(\|W_2\|^2_{\mathcal{W}_x^{2,2}}*\varphi_{2\sigma})\Bigg\|_{\mathcal{W}_t(L^{u'},L^{2})} \\
				&\lesssim \|W_1\|^2_{\mathcal{W}_t(L^{2u},L^{4})\mathcal{W}_x^{2,2}}\|W_2\|^2_{\mathcal{W}_t(L^{2u},L^{4})\mathcal{W}_x^{2,2}}.
			\end{align}	
			The last inequality is due to Lemma \ref{youngtypeinequality}, where $u=(\frac{1}{2\sigma})'$ and $0<2\sigma <\frac{1}{2}$. In view of the relations \(\sigma = \frac{d}{2} - \lambda\) and \(\frac{1}{u} = 1 - 2\sigma\), we obtain \(u = \frac{1}{2\lambda}\). From the inclusion property (ii) of \eqref{waspaceproperties}, we have 
			\begin{align}
				\|W_1 T_{\lambda+ia,\varepsilon} W_2 \|_{\mathfrak{S}^2(L^2(\mathbb{R}^{1+1}))}\lesssim\|W_1\|_{\mathcal{W}_t(L^{\frac{2}{\lambda}},L^{4})(\mathbb{R},\mathcal{W}_x^{2,2}(\mathbb{R}))}\|W_2\|_{\mathcal{W}_t(L^{\frac{2}{\lambda}},L^{4})(\mathbb{R},\mathcal{W}_x^{2,2}(\mathbb{R}))}.\label{H-schattern2norm}
			\end{align}
			For  $\operatorname{Re}(z)=-1,$  following the same steps as in the proof of Theorem \ref{theoremWTWHA}, we have the   following inequality
			\begin{align}\nonumber
				\|W_1T_{z,\epsilon}W_2\|_{\mathfrak{S}^{\infty}(L^2(\mathbb{R}^{1+1}))}&\lesssim  \| W_1\|_{\mathcal{W}_t({L^{\infty},L^{\infty}})(\mathbb{R},\mathcal{W}^{1, \infty}(\mathbb{R})} \| W_2\|_{\mathcal{W}_t({L^{\infty},L^{\infty}})(\mathbb{R},\mathcal{W}^{1, \infty}(\mathbb{R})}\\
				&\lesssim  \| W_1\|_{\mathcal{W}_t({L^{\infty},L^{4}})(\mathbb{R},\mathcal{W}^{1, \infty}\mathbb{R})} \| W_2\|_{\mathcal{W}_t({L^{\infty},L^{4}})(\mathbb{R},\mathcal{W}^{1, \infty}\mathbb{R})} .\label{H-infinityshcatternnorm}
			\end{align}
			The last inequality is obvious from the inclusion property of the Proposition \ref{waspaceproperties}. Using the complex interpolation between \eqref{H-schattern2norm} and \eqref{H-infinityshcatternnorm} followed by letting $\epsilon \to 0^+$, we complete the proof.  
		\end{proof}
		\begin{proof}[Proof of the theorem \ref{theoremONSH-}]
			To prove this Theorem, we shall follow the same theme as Theorem \ref{theoremONSHA} and \ref{ONSH+improved}. From Theorem \ref{WTWtheoremstatH-} and the duality principle (Theorem \ref{DualityWAspace}), we get the estimate 
			\begin{align*}
				\Bigg\|\sum_{j\in J}\alpha_j |e^{it\mathcal{H}^{-}}f_j|^2\Bigg\|_{\mathcal{W}_t(L^{\frac{q}{2}},L^2)(\mathbb{R},  \mathcal{W}_x^{p',p}(\mathbb{R}))} \leq C_{d,p}\|\gamma\|_{\mathfrak{S}^{\frac{2p}{p+1}}(L^2(\mathbb{R}))}.
			\end{align*}
			for $3<p<5$ and $\frac{2}{q}+\frac{1}{p}=1$. Since we have the trivial estimate 
			\begin{align*}
				\Bigg\|\sum_{j\in J}\alpha_j |e^{it\mathcal{H}^{-}}f_j|^2\Bigg\|_{\mathcal{W}_t(L^{{\infty}},L^2)(\mathbb{R},  \mathcal{W}_x^{\infty,1}(\mathbb{R}))} \leq \sum_{j\in J}|\alpha_j|=\|\gamma\|_{\mathfrak{S}^{1}(L^2(\mathbb{R}))},
			\end{align*}
			Interpolation between above two estimate will give us the desired result.

		\end{proof}
		
		%\begin{remark}
		%In Theorem \ref{strichartzH+improved}, we sharpen the Strichartz estimate by raising the time exponent $q/2$ to $q$. We can follow a similar method for operators $\mathcal{H}_{\mathcal{A}}$ and $\mathcal{H}^-$ to obtain an estimate like \eqref{strichartzforH^+} (also \eqref{ONShermite1}), where the time variable is measured in the Lebesgue space instead of Wiener amalgam space. This would be an improved result in the case of $\mathcal{H}_{\mathcal{A}}$, since $L^q=\mathcal{W}(L^q,L^q)\hookrightarrow \mathcal{W}(L^{q/2},L^q)$ for every $2\leq q\leq \infty$. Also, the corresponding orthonormal Strichartz estimate will hold for all dimension $d\geq 1$.  With the similar arguments as in the Proposition \ref{optimalityschattenexponent} and \ref{endpointproposition}, one can also prove the optimality and end point results for the operator $\mathcal{H}_{\mathcal{A}}$.
		%\end{remark}%

		\begin{remark}
			Following a similar argument as in Theorem \ref{strichartzH+improved}, we can get the following version of the Strichartz estimates for the operator $\mathcal{H}^-$
			\begin{equation}
				\|e^{it\mathcal{H}^-} f\|_{L^q \mathcal{W}^{p',p}_x} \lesssim \|f\|_{L^2_x},\label{strichartzforHA1}
			\end{equation}
			with $\frac{2}{q}+\frac{d}{p}=\frac{d}{2}$ and $2\leq p <\frac{2(d+1)}{d-1}.$ Note that in the estimate \eqref{strichartzforHA1}, the time is measured in the Lebesgue space instead of the Wiener amalgam space. Consequently,  we can also obtain the orthonormal estimate corresponding to  \eqref{strichartzforHA1} in every dimension $d\geq 1.$  
		\end{remark}

		\section{Applications to Well-posedness}\label{application}
		As an application of the orthonormal extension discussed in the subsequent sections discussed above, in this section  we investigate the well-posedness of the following operator-valued evolutionary equation of the form:
		\[
		\begin{cases}
			i \partial_t \gamma = [L + w * \rho_\gamma, \gamma], \\
			\gamma|_{t=0} = \gamma_0,
		\end{cases}
		\]
		where the operator $L$ may be any one of $\mathcal{H}^+,\mathcal{H}^-$ or $\mathcal{H}_{\mathcal{A}}$. Note that, both \(\gamma_0\) and \(\gamma = \gamma(t)\) are bounded, self-adjoint operators acting on \(L^2( \mathbb{R}^{d})\). Since the arguments for $\mathcal{H}^+,\mathcal{H}^-$ or $\mathcal{H}_{\mathcal{A}}$, are essentially the same, we provide the proof in detail only for  $ \mathcal{H}^+$. The corresponding results for   
		$ \mathcal{H}^-$ and $\mathcal{H}_{\mathcal{A}}$ are stated without proof.

		Now consider the following equation
		\begin{align}\label{operatorschrodinger}
			\begin{cases}
				i\partial_t{\gamma}(t) = [-\mathcal{H}^+, \gamma(t)] + iR(t), \\
				\gamma(t_0) = 0
			\end{cases}
		\end{align}

		where \(R(t)\) is a self-adjoint operator on \(L^2(\mathbb{R}^d)\)  and   is bounded for almost every \(t\). Then the solution of the above system can be written as
		\begin{align}\label{solutionschrodinger}
			\gamma(t) = \int_{t_0}^{t} e^{i(t-s)\mathcal{H}^+} R(s) e^{i(s-t)\mathcal{H}^+} \, ds.
		\end{align}
		
		The following two inhomogeneous Strichartz estimate plays a crucial role in order to study the well-posedness of the system \eqref{operatorschrodinger}.
		\begin{theorem}[Linear inhomogeneous Strichartz estimate]\label{linearhomostri}
			Assume that \(p, q, d \ge 1\) satisfy
			\[
			1 \leq p <  \frac{d+1}{d-1}, \qquad \frac{2}{q} + \frac{d}{p} = d,
			\]
			and let \(\gamma(t)\) be given by \eqref{solutionschrodinger}. Then
			\[
			\|{ \rho_{\gamma(t)} }\|_{L^q_t(\mathcal{I}, \mathcal{W}^{p',p}_x(\mathbb{R}^d))}
			\le C \left\| \int_{\mathbb{R}} e^{-is\mathcal{H}^+} |R(s)| e^{is\mathcal{H}^+} \, ds \right\|_{\mathfrak{S}^{\frac{2p}{p+1}}},
			\]
			for a constant \(C\) which is independent of \(t_0\).
		\end{theorem}
		
		The proof of the above theorem  follows the same steps as in the proof of the Corollary 1 in \cite{frank}. 
		
		\begin{corollary}[Non-linear inhomogeneous Strichartz estimate]\label{inhomostriestimatenonlinear}
			Let \(d \ge 1\), \(1 \leq p <  \frac{d+1}{d-1}\), and \(q \ge 1\) such that
			\[
			\frac{2}{q} + \frac{d}{p} = d,
			\]
			and let \(\gamma_0 \in \mathfrak{S}^{\frac{2p}{p+1}}\). Let \(\gamma = \gamma(t)\) be the solution to the equation
			
			\begin{align*}
				\begin{cases}
					i \partial_t \gamma = [-\mathcal{H}^+, \gamma] +i R(t), \\
					\gamma|_{t=0} = \gamma_0.
				\end{cases}
			\end{align*}
			Then, the inequality
			\[
			\|{ \rho_{\gamma(t)} }\|_{L^q_t(\mathcal{I}, \mathcal{W}^{p',p}_x(\mathbb{R}^d)))}
			\le C_{\mathrm{Stri}} \left( \| \gamma_0 \|_{\mathfrak{S}^{\frac{2p}{p+1}}}
			+ \left\| \int_{\mathbb{R}} e^{-is\mathcal{H}^+} |R(s)| e^{is\mathcal{H}^+} \, ds \right\|_{\mathfrak{S}^{\frac{2p}{p+1}}} \right)
			\]
			holds for some constant \(C_{\mathrm{Stri}} > 0\) independent of \(\gamma_0\) and \(R\).
		\end{corollary}
		Now using the above inhomogeneous Strichartz estimates, we have the following well-posedness result for the operator valued Hartree equation \eqref{operatorschrodinger}.
		\begin{theorem}[Well-posedness corresponding to $\mathcal{H}^+$]
			Let \(w \in \mathcal{W}^{p',p}(\mathbb{R}^d)\) and \(p, q, d \ge 1\) satisfy
			\[
			1 \le p < \frac{d+1}{d-1}, \qquad \frac{2}{q} + \frac{d}{p} = d.
			\]
			\begin{enumerate}
				\item[(1)](Local well-posedness) For any \(\gamma_0 \in \mathfrak{S}^{\frac{2p}{p+1}}\) with $ R = \| \gamma_0 \|_{\mathfrak{S}^{\frac{2p}{p+1}}} < \infty,$
				there exists $T = T(R, \|w\|_{\mathcal{W}^{p',p}(\mathbb{R}^d)}$\()\) (without loss of generality, we can assume \(T < T_0\)) and a unique
				\[
				\gamma \in C^0([0, T], \mathfrak{S}^{\frac{2p}{p+1}})
				\]
				satisfying \(\rho_\gamma \in L^q([0, T], \mathcal{W}^{p',p}(\mathbb{R}^d)))\) and
				
				\begin{align}\label{schrodingerHoperatorvalued}
					\begin{cases}
						i \partial_t \gamma = [-\mathcal{H}^+ + w * \rho_\gamma, \gamma], \\
						\gamma|_{t=0} = \gamma_0.
					\end{cases} 
				\end{align}
				
				\item[(2)](Almost global well-posedness) For each \(T > 0\) and for any \(\gamma_0 \in \mathfrak{S}^{\frac{2p}{p+1}}\) with
				\[
				\| \gamma_0 \|_{\mathfrak{S}^{\frac{2p}{p+1}}} \le R_T,
				\]
				where \(R_T = R_T(\|w\|_{\mathcal{W}^{p',p}_x(\mathbb{R}^d)})\), then there exists a solution
				\[
				\gamma \in C^0([0, T], \mathfrak{S}^{\frac{2p}{p+1}})
				\]
				satisfying \eqref{schrodingerHoperatorvalued} and \(\rho_\gamma \in L^q([0, T], \mathcal{W}^{p',p}_x(\mathbb{R}^d))\).
			\end{enumerate}
			
		\end{theorem}
		
		\begin{proof}
			
			We start with the proof of the Local well-posedness. Let \(R > 0\) such that
			\[
			\| \gamma_0 \|_{\mathfrak{S}^{\frac{2p}{p+1}}} \le R,
			\]
			and let \(T = T(R) > 0\) be chosen later on. We define a map
			\[
			\Phi(\gamma, \rho) = \left( \Phi_1(\gamma, \rho), \, \rho[\Phi_1(\gamma, \rho)] \right),
			\]
			and for a suitable value of \(T\), we shall show that the map \(\Phi\) is a contraction on the space
			\[
			X := \left\{ (\gamma, \rho) \in C^0_t([0, T], \mathfrak{S}^{\frac{2q}{q+1}}) \times L^q_t([0, T], L^p_x(\mathbb{R}^d)) : \sup_{t\in[0,T]}\| \gamma (t)\|_{\mathfrak{S}^{\frac{2p}{p+1}}} + \| \rho \|_{L^q_t \mathcal{W}^{p',p}_x} \le C R \right\},
			\]
			where $C$ is some positive constant depending on $C_{\mathrm{Stri}}$. The map \(\Phi_1\) is defined as
			
			\begin{align}\label{Phi1definition}
				\Phi_1(\gamma, \rho)(t) = e^{it\mathcal{H}^+} \gamma_0 e^{-it\mathcal{H}^+}
				- i \int_0^t e^{i(t-s)\mathcal{H}^+} [w * \rho(s), \gamma(s)] e^{i(s-t)\mathcal{H}^+} \, ds,
			\end{align}
			for all $(\gamma,\rho)\in X$. Our first aim is to show that the map $\Phi$ is a well defined map. For this, we shall show that 
			$\sup_{t\in [0,T]}\|\Phi_1(\gamma, \rho)(t)\|_{\mathfrak{S}^{\frac{2p}{p+1}}}+\|\rho[{\Phi_1}(\gamma, \rho)]\|_{L^q_t \mathcal{W}^{p',p}_x}\leq C R.$ We start with applying $\mathfrak{S}^{2p/(p+1)}$ norm both sides of equality \eqref{Phi1definition},
			\begin{align}\nonumber
				\|\Phi_1(\gamma, \rho)(t)\|_{\mathfrak{S}^{\frac{2p}{p+1}}}& \leq R + 2 \int_0^T \| w * \rho(s) \|_{\mathcal{W}^{1,\infty}_x} \| \gamma(s) \|_{\mathfrak{S}^{\frac{2p}{p+1}}} \, ds \\ \nonumber
				& \leq R+ 2 \sup_{s\in [0,T]} \|\gamma(s)\|_{\mathfrak{S}^{\frac{2p}{p+1}}} \int_0^T \| w * \rho(s) \|_{\mathcal{W}^{1,\infty}_x} \,ds \\ 
				& \leq R+ 2 T^{1/q'}\|w\|_{\mathcal{W}_x^{p,p'}}(C R)^2.\label{Phi1inequality}
			\end{align}
			Using corollary \ref{inhomostriestimatenonlinear}, we have the following inequality
			\begin{align}\label{rhophi1inequality}
				\|\rho_{\Phi_1(\gamma,\rho)}\|_{L_t^q(\mathcal{I},\mathcal{W}^{p',p}(\mathbb{R}^d))}\lesssim C_{\mathrm{Stri}} (R+ 2 T^{1/q'}\|w\|_{\mathcal{W}_x^{p,p'}}(C R)^2).
			\end{align}
			Eventually we want to choose $T$ small enough such that 
			\begin{align*}
				(1+C_{\mathrm{Stri}}) (R+ 2 T^{1/q'}\|w\|_{\mathcal{W}_x^{p,p'}}(C R)^2)\leq C R.
			\end{align*}
			Therefore, $\Phi$ is a well defined map on $X$. A similar argument
			shows that $\Phi$ is a contraction on $X$, and thus has a unique fixed point on $X$ which
			is a solution to the Hartree equation on $[0,T]$. To prove the almost global well-posedness, we first fix an arbitrary \(T > 0\). Then from estimates \eqref{Phi1inequality} and \eqref{rhophi1inequality}, we have
			\begin{align*}
				\| (\gamma, \rho) \|_{X_T} \le (1 + C_{\mathrm{Stri}}) \left( \| \gamma_0 \|_{\mathfrak{S}^{\frac{2p}{p+1}}}
				+ T^{\frac{1}{q'}} \| w \|_{\mathcal{W}^{p,p'}(\mathbb{R}^d)} \| (\gamma, \rho) \|_{X_T}^2 \right).
			\end{align*}
			With this consideration in mind, we select \(R_T = R_T(\| w \|_{\mathcal{W}^{p,p'}(\mathbb{R}^d)})\) small enough so that there exists \(M > 0\) satisfying, for every \(y \in [0, M]\),
			\[
			(1 + C_{\mathrm{Stri}}) \left( \| \gamma_0 \|_{\mathfrak{S}^{\frac{2p}{p+1}}}
			+ T^{\frac{1}{q'}} \| w \|_{\mathcal{W}^{p,p'}(\mathbb{R}^d)} y^2 \right) \le M,
			\]
			provided that \(\| \gamma_0 \|_{\mathfrak{S}^{\frac{2p}{p+1}}} \le R_T\). Consequently, upon defining the subspace
			\[
			X_{T, M} := \left\{ (\gamma, \rho) \in X_T : \| (\gamma, \rho) \|_{X_T} \le M \right\},
			\]
			we observe that \(\Phi\) invariantly maps \(X_{T, M}\) into itself. Moreover, by further diminishing \(R_T\), we can demonstrate that \(\Phi\) acts as a contraction on \(X_{T, M}\), following arguments analogous to those presented earlier.
			
			Hence, an application of the Banach fixed point theorem guarantees the existence of a solution
			\[
			\gamma \in C^0([0, T], \mathfrak{S}^{\frac{2p}{p+1}})
			\]
			with \(\rho_\gamma \in L^q([0, T], \mathcal{W}^{p',p}(\mathbb{R}^d))\), for any prescribed \(T > 0\).

		\end{proof}
		In a similar manner we can prove the following results that we state without proof.
		\begin{proposition}[Well-posedness corresponding to $\mathcal{H}_{\mathcal{A}}$]
			Let \(w \in \mathcal{W}^{p',p}(\mathbb{R})\) and \(p, q \ge 1\) satisfy
			\[
			1 \le p < 5, \qquad \frac{2}{q} + \frac{1}{p} = 1.
			\]
			\begin{enumerate}
				\item[(1)] (Local well-posedness) For any \(\gamma_0 \in \mathfrak{S}^{\frac{2p}{p+1}}\) with $ R = \| \gamma_0 \|_{\mathfrak{S}^{\frac{2p}{p+1}}} < \infty,$
				there exists $T = T(R, \|w\|_{\mathcal{W}^{p',p}(\mathbb{R})}$\()\) (without loss of generality, we can assume \(T < T_0\)) and a unique
				\[
				\gamma \in C^0([0, T], \mathfrak{S}^{\frac{2p}{p+1}})
				\]
				satisfying \(\rho_\gamma \in \mathcal{W}(L^{q/2},L^q)([0, T], \mathcal{W}^{p',p}(\mathbb{R})))\) and
				
				\begin{align}\label{schrodingerHoperatorvalued1}
					\begin{cases}
						i \partial_t \gamma = [-\mathcal{H}_{\mathcal{A}} + w * \rho_\gamma, \gamma], \\
						\gamma|_{t=0} = \gamma_0.
					\end{cases} 
				\end{align}
				
				\item[(2)] (Almost global well-posedness) For each \(T > 0\) and for any \(\gamma_0 \in \mathfrak{S}^{\frac{2p}{p+1}}\) with
				\[
				\| \gamma_0 \|_{\mathfrak{S}^{\frac{2p}{p+1}}} \le R_T,
				\]
				where \(R_T = R_T(\|w\|_{\mathcal{W}^{p',p}_x(\mathbb{R}^d)})\), then there exists a solution
				\[
				\gamma \in C^0([0, T], \mathfrak{S}^{\frac{2p}{p+1}})
				\]
				satisfying \eqref{schrodingerHoperatorvalued1} and \(\rho_\gamma \in \mathcal{W}(L^{q/2},L^q)([0, T], \mathcal{W}^{p',p}(\mathbb{R})))\).
			\end{enumerate}
			
		\end{proposition}
		
		\begin{proposition}[Well-posedness corresponding to $\mathcal{H}^-$]
			Let \(w \in \mathcal{W}^{p',p}(\mathbb{R})\) and \(p, q \ge 1\) satisfy
			\[
			1 \le p < 5, \qquad \frac{2}{q} + \frac{1}{p} = 1.
			\]
			\begin{enumerate}
				\item[(1)] (Local well-posedness) For any \(\gamma_0 \in \mathfrak{S}^{\frac{2p}{p+1}}\) with $ R = \| \gamma_0 \|_{\mathfrak{S}^{\frac{2p}{p+1}}} < \infty,$
				there exists $T = T(R, \|w\|_{\mathcal{W}^{p',p}(\mathbb{R})}$\()\) (without loss of generality, we can assume \(T < T_0\)) and a unique
				\[
				\gamma \in C^0([0, T], \mathfrak{S}^{\frac{2p}{p+1}})
				\]
				satisfying \(\rho_\gamma \in \mathcal{W}(L^{q/2},L^2)([0, T], \mathcal{W}^{p',p}(\mathbb{R})))\) and
				
				\begin{align}\label{schrodingerHoperatorvalued2}
					\begin{cases}
						i \partial_t \gamma = [-\mathcal{H}^{-} + w * \rho_\gamma, \gamma], \\
						\gamma|_{t=0} = \gamma_0.
					\end{cases} 
				\end{align}
				
				\item[(2)] (Almost global well-posedness) For each \(T > 0\) and for any \(\gamma_0 \in \mathfrak{S}^{\frac{2p}{p+1}}\) with
				\[
				\| \gamma_0 \|_{\mathfrak{S}^{\frac{2p}{p+1}}} \le R_T,
				\]
				where \(R_T = R_T(\|w\|_{\mathcal{W}^{p',p}_x(\mathbb{R}^d)})\), then there exists a solution
				\[
				\gamma \in C^0([0, T], \mathfrak{S}^{\frac{2p}{p+1}})
				\]
				satisfying \eqref{schrodingerHoperatorvalued2} and \(\rho_\gamma \in \mathcal{W}(L^{q/2},L^2)([0, T], \mathcal{W}^{p',p}(\mathbb{R})))\).
			\end{enumerate}
			
		\end{proposition}

		\section{Remarks on some open questions}
		We close by drawing attention to a few directions that, from our perspective, merit further investigation.
		
		\textbf{Restriction on dimension and range:}     For the orthonormal Strichartz inequality for $\mathcal{H}_A$ and $\mathcal{H}^{-}$ proved in  Theorem \ref{theoremONSHA} and in Theorem  \ref{theoremONSH-}, respectively, holds  for dimension $d=1$ only.  A natural  question is whether one can prove these orthonormal ineqality for higher dimension case.  Also, since  we use weak decay of the kernel in the proof of ONS estimate for $\mathcal{H}^-$ and $\mathcal{H}_{\mathcal{A}}$, we expect the range of exponent $p$ in Theorem \ref{theoremONSHA} and Theorem \ref{theoremONSH-} can be improved.

		\textbf{Restriction problem:} The restriction problem has a delicate connection to many other conjectures, notably the Kakeyaand Bochner-Riesz conjectures.   			
		Furthermore, it also closely related to that of estimating solutions to linear PDE such as the wave and Schr\"odinger equations;   ands this connection was first observed by  Strichartz in \cite{SSIN}.  More specically, Strichartz inequalities for the Schr\"odinger and wave equations correspond to Fourier restriction estimates on the paraboloid and the cone surfaces, respectively. So one of the open direction is to investigate restrictiuon problem on Winer alamgam spaces.

		Let $S \subset \mathbb{R}^d, d\geq 2$ be a hyper-surface endowed with Lebesgue measure $d\mu$.  We define the restriction operator 
		$ 
		\mathcal{R}_{S}: \mathcal{W}^{p_1,p_2}(\mathbb{R}^d)\to S 
		$
		defined as 
		\begin{align} \label{restricitonpro}
			\mathcal{R}_{S}(F)=\hat{F}|_{S}.
		\end{align}
		Here    $\hat{F}$ denotes the Fourier transform for a   Schwartz class function $F$ on $\mathbb{R}^d$ and  is defined as
		$$\hat{F}(\xi)=\frac{1}{(2\pi)^{d/2}}\int_{\mathbb{R}^d}F(x)e^{-ix\cdot\xi}dx.$$
		
		Since $ \mathcal{W}^{1,1}(\mathbb{R}^d)$ is embedded inside $L^1(\mathbb{R}^{d})$, hence $\hat{F}$ is continuous if function
		$F\in  \mathcal{W}^{1,1}(\mathbb{R}^d)$. Therefore, the operator $\mathcal{R}_{S}$ is well defined for $ p_1=p_2=1$. However, $ \mathcal{W}^{2,2}(\mathbb{R}^d)=L^2(\mathbb{R}^{d})$ implies $\mathcal{R}_{S}$ cannot be meaningfully defined for $p_1=p_2=2$. Since,   the literature on $ \mathcal{R}_{S}$ is extensive, to the best our knowledge, we are not aware of any results addressing: for what values of $1\leq  p_1,p_2 \leq 2$ is the operator $\mathcal{R}_{S}$, well defined?  We therefore pose  the following restriction  problem:

		\begin{problem}
			For what values of $1\leq  p_1,p_2 < 2$, does the Fourier transform of a function $f \in \mathcal{W}^{p_1,p_2}  (\mathbb{R}^d)$ belong to $L^{2}(S)$, where $S$ is endowed with its $(d-1)$-dimensional Lebesgue measure $d \sigma$? 
		\end{problem}

		The operator dual to \(\mathcal{R}_S\) is called the extension operator, which we denote by \(\mathcal{E}_S\), and satisfies the identity
		\begin{equation}
			\mathcal{E}_S f(x) = \frac{1}{(2\pi)^{d/2}} \int_S f(\xi) e^{i\xi \cdot x} \, d\sigma(\xi), \qquad \forall x \in \mathbb{R}^d, 
		\end{equation}
		for all \(f \in L^1(S)\). The restriction problem is equivalent to knowing when \(\mathcal{E}_S\) is bounded from  \(L^{2}(S)\) to \(\mathcal{W}^{p_1',p_2'}(\mathbb{R}^d) \). 
		Moreover,  \(\mathcal{E}_S\) is bounded from \(L^2(S)\) to \( \mathcal{W}^{p_1',p_2'}(\mathbb{R}^d)\) if and only if the operator \(T_S := \mathcal{E}_S (\mathcal{E}_S)^*\) is bounded from \( \mathcal{W}^{p_1',p_2'}(\mathbb{R}^d)\) to 
		\( \mathcal{W}^{p_1,p_2}(\mathbb{R}^d)\).

		\section*{Acknowledgments} 
		This work is a part of Sarthak Tiwari's Ph. D. thesis. The first and third authors acknowledge the support of the National Institute of Science Education and Research (NISER) and the Homi Bhabha National Institute (HBNI) for providing excellent research facilities. RM acknowledges the partial support provided by the research grants  (ANRF/ARGM/2025/000750/MTR). SSM is supported by the DST-INSPIRE Faculty Fellowship DST/INSPIRE/04/2023/002038. 
		
		\par
		
		\noindent{\bf Data Availability.}
		Data sharing not applicable to this article as no datasets were generated or analysed during the current study.
		
		\par
		
		\noindent{\bf Conflict of interest.}
		The authors declare that there is no conflict of interest.

\end{document}